\documentclass[11pt,reqno]{amsart}

\usepackage{comment}
\usepackage{ragged2e}
\usepackage{mathtools}
\usepackage{amsmath,amssymb,amsthm,mathrsfs}
\usepackage{xcolor}
\usepackage{amsfonts}
\usepackage[english]{babel}
\usepackage[utf8x]{inputenc}

\newcommand*{\bigchi}{\mbox{\Large$\chi$}}

\newcounter{lemmac}
\newcounter{theoremc}
\newcounter{corollc}
\newcounter{propc}
\newcounter{rmkc}

\newtheorem{theorem}[theoremc]{Theorem}
\newtheorem*{theorem*}{Theorem}
\newtheorem*{defn*}{Definition}
\newtheorem*{prob*}{Problem}
\newtheorem*{lemma*}{Lemma}
\newtheorem*{conj*}{Conjecture}
\newtheorem*{corol*}{Corollary}
\newtheorem*{conclusion*}{Conclusion}
\newtheorem*{proposition*}{Proposition}
\newtheorem{lemma}[lemmac]{Lemma}
\newtheorem{proposition}[propc]{Proposition}
\newtheorem{corollary}[corollc]{Corollary}
\newtheorem{rmk}[rmkc]{\textit{Remark}}
\newtheorem*{rmk*}{\textit{Remark}}

\title[Branching Rules for Classical Groups]{Relative Weyl Character formula, Relative Pieri formulas and Branching rules for Classical groups}

\author{C.S. Rajan}
\address{Ashoka University, Sonipat, Haryana 131029, India}
\email{cs.rajan@ashoka.edu.in}

\author[S.~Shrivastava]{Sagar Shrivastava}
\address{School of mathematics, 
Tata Institute of Fundamental Research, 
Homi Bhabha Road, Mumbai 400005, India}
\email{sagars@math.tifr.res.in}

\begin{document}

\maketitle

\begin{abstract} We give alternate proofs of the classical branching rules for highest weight representations of a complex reductive group $G$ restricted to a closed regular reductive subgroup $H$, where $(G,H)$ consist of the pairs  $(GL(n+1),GL(n))$, $ (Spin(2n+1), Spin(2n)) $ and $(Sp(2n),Sp(2)\times Sp(2n-2))$. Our proof is essentially a long division. The starting point is a relative Weyl character formula and our method is an inductive application of a relative Pieri formula. We also give a proof of the branching rule for the case of $ (Spin(2n), Spin(2n-1))$, by a reduction to the case of $(GL(n),GL(n-1))$.
\end{abstract}
 
\section{Introduction}

Let $G$ be a connected reductive algebraic group over $\mathbb{C}$ and $H$ be a closed, connected reductive subgroup of $G$. In this article, we consider the branching problem, that of understanding the restriction of a finite dimensional rational representation $\pi$ of $G$ to $H$. For a dominant weight $\lambda$ of $G$, let $\pi_{\lambda}$ be a corresponding irreducible representation of $G$ with highest weight $\lambda$ and $\bigchi_{\lambda}$ be its character. By complete reducibility of representations of $H$, 
\begin{equation}\label{eqbrmu}
\bigchi_{\lambda}|_H = \sum_{\mu} m(\lambda,\mu)\bigchi_{\mu} ,
\end{equation}
where $m(\lambda,\mu)$ is the multiplicity of the the irreducible highest weight representation $\pi_{\mu}$ of $H$ (corresponding the the weight $\mu$ of $H$) in $\pi_{\lambda}|_H$. The branching problem is to understand $m(\lambda,\mu)$ as a function of $\lambda$ and $\mu$.

These have a rich and classical history, going back to the work of Herman Weyl \cite{weyl1950theory}, where he considers $G=GL(n+1) $ and $H=GL(n) $.
The starting point for a study of branching rules is the Weyl character formula, which gives a formula for the multiplicity of the weights with respect to a maximal torus $T$ of $G$, occurring in an irreducible representation $\pi$ of $G$. A systematic approach to the branching rules is to make use of Kostant's partition function, which effectively corresponds to inverting the denominator appearing in the Weyl character formula. We refer to the books of \cite{goodman1998} and \cite{knappliegroups2002}, for a more detailed exposition of the classical branching rules.

In this paper, we modify this approach. 
As an example, let us consider the branching from $GL(n+1)$ to $GL(n)$. A highest weight $\lambda$ for $GL(n+1)$ is given by an $(n+1)$-tuple $\lambda= (\lambda_1, \lambda_2  \ldots, \lambda_{n+1}) \in \mathbb{Z}^{n+1}$ with $\lambda_1 \geq \lambda_2 \geq \ldots \geq \lambda_{n+1} $. Weyl's branching 
formula  is the following theorem:
\begin{theorem} \label{theorem:branching-gln}
   Let $\lambda$ be a dominant weight of $G=GL(n+1) $ and $\mu$ a dominant weight of $H=GL(n) $. Then $m(\lambda,\mu) =1$ if  $\mu$ interlaces $\lambda$ ($\mu \preceq \lambda$) i.e. $\lambda_1 \geq \mu_1 \geq \lambda_2 \geq \ldots \geq \mu_{n} \geq \lambda_{n+1}$, and zero otherwise.
\end{theorem}
Our starting point for the proof of the foregoing theorem is to consider a relative Weyl character formula. 
We take for $T$, the maximal torus consisting of diagonal matrices $g$ with diagonal entries $x_1,x_2,\ldots,x_{n+1}$. The Schur-Weyl character formula corresponding to the highest weight $\lambda$, gives the formula for the character $\bigchi_{\lambda}$ restricted to $T$:
\begin{equation}\label{eq:weylchformulagldet}
\bigchi_{\lambda}(g) = \frac{\det |x^{\lambda_i+n+1-i}_j|}{\det |x^{n+1-i}_j|}. 
\end{equation}

The relative Weyl character formula we consider for the branching rules for $(G,H) = (GL(n+1) ,GL(n) \times GL(1) )$  is the co-factor expansion of the above determinantal formula for the character, leading to a rational expression in one variable with coefficients characters of $GL(n)$. Such a co-factor expansion of the Schur-Weyl character formula was used in \cite[4.1]{rajan2014}. For $\mu = (\mu_1, \mu_2, \ldots, \mu_{n+1}) \in \mathbb{Z}^{n+1}$, where $ \mu_1 > \mu_2 > \cdots > \mu_{n+1}$, let 
$$S(\mu) = \det |x^{\mu_i}_j| .$$
With this notation,  $\bigchi_{\lambda} = S(\lambda +\rho_{n+1})/S(\rho_{n+1})$, where $\rho_{n+1} = (n,n-1,\ldots,0)$. The co-factor expansion for the determinant $S(\mu)$  is given by:
$$(-1)^{n+1}S(\mu) = x_{n+1}^{\mu_{n+1}} S(\mu^{(n+1)})  - x_{n+1}^{\mu_{n}} S(\mu^{(n)})  + \ldots + (-1)^{n} x_{n+1}^{\mu_1} S(\mu^{(1)}) ,  $$
where $\mu^{(i)} = (\mu_1+1,\mu_2+1, \ldots, \mu_{i-1}+1, \mu_{i+1}, \ldots, \mu_{n+1})$ (by an abuse of notation, we use $S(\mu)$ in $n$ variables as well). Upon substituting $t=x_{n+1}$, the Schur-Weyl character formula becomes

$$\bigchi_{\lambda} = \frac{t^{\lambda_{n+1}} S(\lambda^{(n+1)} + \rho_{n}) - t^{\lambda_{n}+1} S(\lambda^{(n)} + \rho_{n}) + \ldots + (-1)^{n} t^{\lambda_{1}+n} S(\lambda^{(1)} + \rho_{n})}{S(\omega_{n} + \rho_{n}) - t S(\omega_{n-1} +\rho_{n}) + \ldots + (-1)^{n} t^{n} S(\rho_{n}) },$$
where $\omega_k=(1,1,1,\ldots,1,0,\ldots ,0)$ ($k$ many 1's) corresponds to the $k^{th}$ fundamental weight for $GL(n)$. In the above expression, we can divide the numerator and denominator by the Weyl denominator $S(\rho_{n})$ of $GL(n)$, to get

\begin{equation}\label{eq:relweylchformulagln1}
\bigchi_{\lambda}= \frac{t^{\lambda_{n+1}} \bigchi_{\lambda^{(n+1)}}  + \ldots + (-1)^{n} t^{\lambda_{1}+n} \bigchi_{\lambda^{(1)}}}{ \bigchi_{\omega_{n}} + \ldots + (-1)^{n} t^{n} }.
\end{equation}

The above expression is the relative Weyl character formula for $(GL(n+1), GL(n)\times GL(1))$, where we express the character $\bigchi_{\lambda}$, as a rational expression in terms of  the irreducible characters of $GL(n)\times GL(1)$. The denominator (resp. numerator) in the relative Weyl character formula would be called the relative Weyl denominator (resp. numerator) for the pair $(GL(n+1),GL(n) \times GL(1))$.

In order to obtain the branching laws, we essentially carry out a long division of the relative Weyl character formula. For this, we need to understand a relative Pieri formula (see Proposition \ref{RelativePieri:GLn}), i.e., to understand the tensor product decomposition of an irreducible representation of $GL(n) $, with the relative Weyl denominator. In this case, this amounts to understanding the tensor product decomposition of an irreducible representation of $GL(n)$ with all of the fundamental representations, which are exterior powers of the standard representation of $GL(n)$. 

In order to carry out the long division, the relative Pieri is used in an inductive manner. This leads to some  combinatorial identities which then yield the branching laws in an inductive manner. We refer to Section 2 for the details of the proof.

\subsection{Classical branching: regular case.} In the case of orthogonal groups, the branching rules for $(Spin(n+1),Spin(n))$ were first proved by F.D.Murnaghan \cite[Ch-IX]{murnaghan1938} (see also \cite{zelo1973}). We recall the branching rule 
when $H$ is regular in $G$, i.e., when the ranks of $G$ and $H$ are equal:

\begin{theorem} \label{thmbranchingbn}\label{theorem:branching-orth-reg}
   Let $\lambda$ be a dominant weight of $G=Spin(2n+1)$ and $\mu$ is a dominant weight of $H=Spin(2n)$. Then $m(\lambda,\mu) =1$ if $\mu$ interlaces $\lambda$ ($\mu \preceq \lambda$) i.e. $\lambda_1 \geq \mu_1 \geq \lambda_2 \geq \ldots \geq \mu_{n-1} \geq \lambda_n \geq |\mu_n|$, and zero otherwise.
\end{theorem}

Branching rules in the symplectic case for $(Sp(2n),Sp(2)\times Sp(2n-2))$  have a fairly long history (see  \cite{zelo1973}, \cite{hegerfeldt1967}, \cite{miller1966branching} and  \cite{lepowsky1970representations}). A description of the multiplicities in terms of $SL_2$ representations was given by N.R.Wallach  and O. Yacobi \cite{wallach2009multiplicity}.
\begin{theorem} \label{thmbranchingcn}\label{theorem:branching-symplectic}
Let $G= Sp(2n)$ and $H= Sp(2)\times Sp(2n-2)$. Let $\lambda$ be a dominant weight for $Sp(2n) $ and  $\mu$ is a dominant weight of $Sp(2n-2) $. Denote by $S^{(k)}$ the $(k+1)$-dimensional irreducible representation of $SL_2$ (isomorphic to $Sp(2)$). Let
$$\bigchi_{\lambda}|_{H} = \sum_{\mu} \sum_k m(\lambda, \mu, k) S^{(k)} \bigchi_{\mu}.$$
Then representation $V(\lambda, \mu) :=\sum_k m(\lambda, \mu, k) S^{(k)} $ of $SL_2$ is non zero if and only if 
$$\lambda_j \geq \mu_j \geq \lambda_{j+2} ,\, \, \text{ for } 1 \leq j \leq n-1$$
(here $\lambda_{n+1}=0$). When the inequalities are satisfied, let 
$$x_1 \geq y_1 \geq x_2 \geq y_2 \geq \cdots \geq x_n \geq y_n,$$
be the non decreasing rearrangement of $\{\lambda_1, \ldots, \lambda_n, \mu_1, \ldots, \mu_{n-1},0 \}$. Then as $SL_2$-modules,
$$V(\lambda, \mu) = \bigotimes_{i=1}^n S^{(x_i -y_i)}.$$
\end{theorem}

In both the foregoing theorems, the proof proceeds in a manner similar to that of $GL(n)$.  A relative Weyl character formula can be obtained directly from the determinantal versions of the Weyl character formula. More generally, in the case when $H$ is regular in $G$, a relative Weyl character formula has been obtained in \cite{gross1998weyl} (see Section 3). When $H$ is the Levi component of a parabolic, a relative Weyl character formula is given in \cite{kostant61}.

To obtain relative Pieri formulas, we follow a method of \cite{Okada16}. In this, rather than working out the individual tensor product decompositions of an irreducible representation of $H$ with that of a representation `occurring' in the relative Weyl denominator, we use the determinental expression for the Weyl denominator to derive a `weaker'  relative Pieri formula, that of decomposing the tensor product of an irreducible representation of $H$ with the full virtual representation given by the relative Weyl denominator. 

From the relative Pieri formula, following the process of long division, the branching rules are derived in an inductive manner. We first observe that there is a unique term in the relative Weyl denominator, which is the largest (or the smallest) with respect to a suitable ordering (for example, either the degree in the case of $GL(n+1)$, or in general the lexicographic ordering). Upon cross multiplying the restriction of the character $\bigchi_{\lambda}$ to $H$ by the relative Weyl denominator and making use of the relative Pieri formula, we get an expression for the multiplicity with which a dominant weight $\mu$ of $H$ occurs in $\bigchi_{\lambda}$,  in terms of multiplicities of dominant weights which are larger (smaller) in the ordering. By an induction hypothesis, these multiplicities are as expected. This leads to a combintorial expression, which can be solved to proceed inductively with a proof of the branching formula. This is essentially, the process of long division. 

The proof proceeds by first establishing a weak  interlacing property. This property limits the dominant weights $\mu$ of $H$ to be considered to be amongst the expected weights occurring in the branching formula, together   with the weights occurring in the relative  Weyl numerator and some boundary cases. A further analysis of these three cases, yields the proof of the branching laws. We would like to reiterate that other than the determinantal formulas for Weyl character formula, everything else in this paper is self contained.

We will prove Theorem \ref{theorem:branching-orth-reg} in Section 4 and Theorem \ref{theorem:branching-symplectic} in Section 6.

\subsection{Classical branching: non-regular case.} For the non-regular case of $(Spin(2n),Spin(2n-1))$, we do not have at our disposal a relative Weyl character formula. Instead we express the  Weyl character formula as a formal sum of Weyl character type formula for $GL(n) $, and then apply branching rules for $(GL(n) ,GL(n-1) )$  to get the desired branching rules for $(Spin(2n),Spin(2n-1))$. We give a proof of the 
branching theorem in Section 5:
\begin{theorem} \label{thmbranchingdn}
   Let $\lambda$ be a dominant weight of $G=Spin(2n)$ and $\mu$ is a dominant weight of $H=Spin(2n-1)$. Then $m(\lambda,\mu) =1$ if $\mu$ interlaces $\lambda$ ($\mu \preceq \lambda$) i.e. $\lambda_1 \geq \mu_1 \geq \lambda_2 \geq \ldots \geq  \mu_{n-1} \geq |\lambda_n| $, and zero otherwise.
\end{theorem}
We remark that the known proofs of the foregoing theorem (\cite{goodman1998})
also make use of the corresponding proof of branching for $GL(n)$.

\section{Branching rules for $(GL(n+1) , GL(n) )$}
In this Section, we give a proof of Theorem \ref{theorem:branching-gln}, following the schema given in the introduction.  The starting point is  the relative Weyl character formula for $(GL(n+1),GL(n))$  given by equation (\ref{eq:relweylchformulagln1}).

\begin{proposition}{(Relative Weyl character formula)}
Let $\lambda = (\lambda_1, \lambda_2, \ldots, \lambda_{n+1})$ be a dominant weight for $G=GL(n+1)$ and $H= GL(1)\times GL(n)$. Let $\bigchi_{\lambda}$ be the irreducible highest weight representation of $G$ with highest weight $\lambda$. We have the following expression for restriction of $\bigchi_{\lambda}$:
\begin{equation}\label{eq:relweylchformulagln}
\bigchi_{\lambda}|_{H}= \frac{t^{\lambda_{n+1}} \bigchi_{\lambda^{(n+1)}}  + \ldots + (-1)^{n} t^{\lambda_{1}+n} \bigchi_{\lambda^{(1)}}}{ \bigchi_{\omega_{n}} + \ldots + (-1)^{n} t^{n} }.
\end{equation}
where 
\begin{itemize}
    \item $t$ is the character of $GL(1)$,
    \item $\omega_k=(1,1,1,\ldots,1,0,\ldots ,0)$ ($k$ many 1's) corresponds to the $k^{th}$ fundamental weight for $GL(n)$, 
    \item $\lambda^{(i)} = (\lambda_1+1,\lambda_2+1, \ldots, \lambda_{i-1}+1, \lambda_{i+1}, \ldots, \lambda_{n+1})$ are dominant weights of $GL(n)$.
\end{itemize}
\end{proposition}

\subsection{Relative Pieri formula}
The second step is to consider a relative Pieri formula. 
Let $\Delta = \bigchi_{\omega_{n}} + \ldots + (-1)^{n} t^{n}$ be the relative Weyl denominator as in equation (\ref{eq:relweylchformulagln}). Our aim now is to obtain a relative Pieri formula, giving the decomposition of $\bigchi_{\mu} \Delta$, for a dominant weight $\mu$ of $GL(n)$.

We see that the representations occurring in the relative Weyl denominator are exactly the fundamental representations of $GL(n)$. Hence the relative Pieri formula is a signed sum of the skew/dual Pieri formula for $GL(n)$ which gives a decomposition of $\bigchi_{\mu} \bigchi_{\omega_i}$. For the proof of relative Pieri, we follow \cite{Okada16}, who proves the dual Pieri formula for the symplectic case. Instead of considering the individual dual Pieri, it is convenient to work with the full relative Weyl denominator. The reason for doing so is the product formula for $\Delta$, given by the following lemma:

\begin{lemma}
With the above notation, the relative Weyl denominator has a product expansion, $\Delta = \prod_i(x_i-t)$.
\end{lemma}
\begin{proof}
The denominator in the Schur-Weyl character formula (equation \ref{eq:weylchformulagldet}) is the Vandermonde determinant, which has a product formula $\prod_{i<j}( x_i -x_j)$. Upon dividing by the Weyl denominator of $GL(n)$, this gives us that the relative Weyl denominator  $\Delta = \prod_i(x_i-t)$.
\end{proof}

We use the following notation for length: for $\xi \in \mathbb{R}^n,$ let  $  |\xi| = \sum_i |\xi_i|$ be the length of $\xi$.

If $\mu$ is not a dominant weight, we take $\bigchi_{\mu} =0$.

Given $\lambda= (\lambda_1, \lambda_2, \ldots, \lambda_{n+1})$ be a dominant weight of $GL(n+1)$ and\\ $\mu = (\mu_1,\mu_2, \ldots, \mu_n)$ be a dominant weight of $GL(n)$. Then $\mu \preceq\lambda$ is used to say that $\mu$ interlaces $\lambda$, i.e. $\lambda_1 \geq \mu_1 \geq \lambda_2 \geq \ldots \geq \mu_{n} \geq \lambda_{n+1}$.

\begin{proposition}[Relative Pieri formula]\label{RelativePieri:GLn}
Let $\nu$ is a dominant weight for $GL(n)$. With the above notation, we have the following tensor product decomposition 
\begin{equation}
\bigchi_{\nu} \Delta = \bigchi_{\nu} \left(\sum_{i=0}^n (-t)^{i} \bigchi_{\omega_{n-i}} \right) =  \sum_{\substack{\varepsilon \in \{0,1 \}^n \\ \nu+ \varepsilon \text{ dominant} }} (-t)^{n-|\varepsilon|}\bigchi_{\nu+ \varepsilon}. 
\end{equation}

\end{proposition}

\begin{proof}
From the Weyl character formula equation (\ref{eq:weylchformulagldet}), we get,
\begin{align*}
\Delta  \bigchi_{\nu}  &= \left( \prod_{k=1}^n(x_k-t) \right)\frac{S(\nu + \rho_n)}{S(\rho_n)}\\
&= \frac{\displaystyle \prod_{k=1}^n(x_k-t) \det |x^{\nu_i + n-i}_j| }{S(\rho_n)}.\\
\end{align*}
Using the multilinearity of the determinant, we bring the factor $(x_i-t)$ to the $i^{th}$ column.

\begin{equation*}
\prod_{k=1}^n(x_k-t)\det
\begin{vmatrix}
  x_1^{\nu_1+n-1}  & \cdots & x_n^{\nu_1+n-1} \\
  x_1^{\nu_2+n-2}  & \cdots & x_n^{\nu_2+n-2} \\
  \vdots    & \ddots & \vdots  \\
  x_1^{\nu_n}  & \cdots & x_n^{\nu_n} 
 \end{vmatrix}
 = 
\det
\begin{vmatrix}
  x_1^{\nu_1+n-1} (x_1-t) & \cdots & x_n^{\nu_1+n-1} (x_n-t) \\
  x_1^{\nu_2+n-2} (x_1-t) &\cdots & x_n^{\nu_2+n-2} (x_n-t)\\
  \vdots    &\ddots & \vdots  \\
  x_1^{\nu_n}(x_1-t)  &\cdots & x_n^{\nu_n} (x_n-t) 
 \end{vmatrix}.
\end{equation*}

The individual terms in the determinant can be expanded to give us 
$$x_j^{\nu_i + n-i+1} - tx_j^{\nu_i + n-i}.$$
Every row is a sum of two rows, with one having a linear factor of $t$. We now expand along the rows. In the rows with the factor of $t$, the exponent of $x_j$ remains invariant, whereas in the row without the factor of $t$, the exponent increases by 1. In a given determinant, if there are $r$ rows that have a $(-t)$ factor, we know that the other $n-r$ rows would have the corresponding exponents incremented by 1. This gives us
\begin{equation}
\Delta  \bigchi_{\nu} =\sum_{\varepsilon \in \{0,1 \}^n } (-t)^{n-|\varepsilon|}\frac{\det |x^{\nu_i + n-i + \varepsilon_i}_j | }{S(\rho_n)}.
\end{equation}
In the above sum, corresponding to $\varepsilon$, the determinant is in a standard form for the numerator in the Weyl character formula for irreducible representation of $GL(n)$, with highest weight $\nu +\varepsilon$. If $\nu +\varepsilon$ is not dominant, then there is an $i$ such that $\nu_i +\varepsilon_i < \nu_{i+1}+ \varepsilon_{i+1}$. This forces $\nu_i = \nu_{i+1}$ and $\varepsilon_i=0, \varepsilon_{i+1}=1$, which tells us that $S(\nu+ \varepsilon + \rho_n)$ has two rows having the same entries, therefore vanishes. Hence the summands are non zero only if $\nu+\varepsilon$ is a dominant weight, which gives us

\begin{equation}
\Delta  \bigchi_{\nu} = \sum_{\substack{\varepsilon \in \{0,1 \}^n \\ \nu+ \varepsilon \text{ dominant} }} (-t)^{n-|\varepsilon|}\bigchi_{\nu+ \varepsilon}.
\end{equation}

This completes the proof of the relative Pieri formula for $GL(n)$.
\end{proof}

\begin{corollary}[Dual Pieri formula]
Let $\nu$ be a dominant weight for $GL(n)$ and $\omega_i$ be the $i^{th}$ fundamental weight, corresponding the $i^{th}$ exterior power of the defining representation. Then
$$\bigchi_{\nu}  \bigchi_{\omega_{i}} =  \sum_{\substack{\varepsilon \in \{0,1 \}^n , |\varepsilon|=i \\ \nu+ \varepsilon \text{ dominant} }} \bigchi_{\nu+\varepsilon}.$$
\end{corollary}

\begin{proof}
    Follows from comparing the graded components of the sum as polynomial in $t$ in the relative Pieri formula.
\end{proof}

\subsection{Proof of $(GL(n+1),GL(n))$ branching rule}


\setcounter{theoremc}{0}
\setcounter{rmkc}{0}

We now prove the branching rule for $(GL(n+1), GL(1) \times GL(n))$.

By equation (\ref{eq:relweylchformulagln}), we see that the character $\bigchi_{\lambda}$ is a homogeneous polynomial in $x_i's$ (and $t=x_{n+1}$) of total degree $|\lambda|$. Hence by equation (\ref{eqbrmu}), we get :

\begin{equation}\label{eqnred}
    \bigchi_{\lambda}|_{H}= \sum_{\nu} m(\nu) \bigchi_{\nu}t^{|\lambda |-|\nu|}= \sum_{r=0}^{\lambda_1+n}t^r \left( \sum_{r + |\nu| = |\lambda|} m(\nu) \bigchi_{\nu} \right).
\end{equation}

By tensoring with a suitable power of the determinant, we can assume that $\lambda_{n+1}=0$. As the numerator and denominator in the relative Weyl character are polynomials in $t$, and the relative Weyl denominator has a non-zero constant term, it follows $\bigchi_{\lambda}|_{H}$ is a polynomial in $t$.

We prove the theorem by increasing induction on the degree of $t$ (we can also prove it by decreasing induction on the degree of $t$). Cross multiplying  equation (\ref{eqnred}) by $\Delta$ and comparing constant coefficients, we get
$$\bigchi_{\lambda^{(n+1)}}= \sum_{|\nu|=|\lambda|} m(\nu) \bigchi_{\nu} \bigchi_{\omega_n} .$$
This gives us, 
\[m(\lambda^{(n+1)} - (1,1,\ldots,1)) = m(\lambda_1,\lambda_2,\ldots, \lambda_n) =1.\]

For our induction hypothesis, we assume that the coefficients of $t^r$ in Equation (\ref{eqnred}) for all $r<k$  satisfy the hypothesis in the theorem. We have shown the induction hypothesis to be true for $k=1$. Note that by homogeneity, the sum of the degree of $t$ and $|\nu|$ is equal the length of $|\lambda|$. By the induction hypothesis, we know that $m(\nu)=1$ for those highest weights $\nu$, for which $|\nu| > |\lambda|-k$, and interlace $\lambda$. We can rewrite the above equation (\ref{eqnred}) as 

\begin{equation}
\bigchi_{\lambda}|_{H}= \sum_{r=0}^{k-1}t^r \left( \sum_{\substack{r + |\nu'| = |\lambda|\\ \nu \preceq \lambda }}  \bigchi_{\nu'} \right) + \sum_{r \geq k}t^r\left( \sum_{r + |\nu| = |\lambda|} m(\nu) \bigchi_{\nu} \right).    
\end{equation}

We would like to understand the multiplicity $m(\nu)$ for a highest weight $\nu$ with $|\nu| +k = |\lambda|$. Upon cross multiplying the above equation by $\Delta$, and using the fact that the constant term of $\Delta$ is $\bigchi_{\omega_n}$, we see that the character $\bigchi_{\nu+\omega_n}$ occurs with multiplicity $m(\nu)$ in the second part of the above sum. 

Applying relative Pieri to the first term in the above expression, we get
$$\sum_{r=0}^{k-1}t^r \left( \sum_{\substack{r + |\nu'| = |\lambda|\\ \nu \preceq \lambda }}  \bigchi_{\nu'} \right) \Delta =\sum_{r=0}^{k-1} t^{r}\left(\sum_{\substack{\nu \preceq \lambda \\r + |\nu'| = |\lambda|}} \sum_{\varepsilon \in \{0,1 \}^n } (-t)^{n-|\varepsilon|}\bigchi_{\nu'+ \varepsilon} \right).$$

We see by homogeneity that a degree $k$ term occurs when $r +n - |\varepsilon|=k$, i.e. $|\lambda|  + n=k+ |\nu'+\varepsilon|$. Note that if $|\varepsilon|=n$, then it can't contribute to the coefficient of $t^k$, as $r<k$. Hence the contribution to the coefficient of $t^k$  term is given by
\begin{equation} \label{eqn:signedsum1}
\sum_{|\varepsilon|<n} (-1)^{n- |\varepsilon|} \left(\sum_{ \substack{|\nu'+\varepsilon|+k=|\lambda|+n\\ \nu' \preceq \lambda}}  \bigchi_{\nu'+\varepsilon} \right).
\end{equation}

To understand the multiplicity $m(\nu)$, we need to know the multiplicity say $n(\nu)$ of  $\bigchi_{\nu + \omega_n}$ in the above expression. By the foregoing equation, this amounts to counting the (signed) number of ways $\nu + \omega_n$ can be written as $\nu' + \varepsilon$ with $|\varepsilon|<n$ and $\nu'$ interlacing $\lambda$. Note that this can be seen as $\nu + \varepsilon'=\nu'$, where $\varepsilon' = \omega_n -\varepsilon$, $|\varepsilon'|=n-|\varepsilon|$ and $|\varepsilon'|>0$.

The key step in the proof of branching is to count the number of ways a given $\nu$ can be modified with $\varepsilon'\in \{0,1\}^n,$ such that $\nu + \varepsilon'$ interlaces $\lambda$. To illustrate the proof, suppose $\nu$ is ``generic", i.e. $\lambda_i > \nu_i \geq \lambda_{i+1}$ for all $i$. In this case, we can arbitrarily assign  $\varepsilon'_i =1$ at any set of $k$-indices with $0<k\leq n$ so that $\nu' = \nu+ \varepsilon'$ interlaces $\lambda$. The total number of such $\varepsilon'$ is given by the binomial coefficient $\binom{n}{|\varepsilon'|}=\binom{n}{k}$. Thus the multiplicity with which the weight 
$\nu+\omega_n$ occurs in equation (\ref{eqn:signedsum1}) is given by, 
\[ n(\nu)=\sum_{k=1}^{n}(-1)^{k}\binom{n}{k}=-1.\]
From the expression for  the relative Weyl numerator, the co-ordinates weights $(\lambda^{j}+\rho_n)_j>\lambda_j$ for $j<i$. Hence the weights $\nu+\omega_n$  do not occur in the Weyl numerator and this gives that $m(\nu)+n(\nu)=0$. Thus for generic $\nu$, we get $m(\nu)=1$.\\

We modify the above argument for general $\nu$. We first show that $\nu$ {\em weakly interlaces} $\lambda$ in the following sense:

\begin{lemma}\label{lemma-almostinterlace}
Let $\nu$ be such that $\nu'  = \nu + \varepsilon'$ interlaces $\lambda$ for some  $\varepsilon' \in \{0,1\}^n,$ with $|\varepsilon'|>0.$  Then  

\begin{equation} \label{eq-weaklyinterlace}
\lambda_i \geq \nu_i \geq \lambda_{i+1}-1.
\end{equation}.
\end{lemma}

\begin{proof}
Given that $\nu'_i=\nu_i + \varepsilon'_i $, and as $\nu' \preceq \lambda$, we get that  $\lambda_i \geq \nu'_i \geq \lambda_{i+1}$. Putting the two of them together, we get that $\lambda_i \geq \nu_i \geq \lambda_{i+1}-1$.
\end{proof}

 Let $M$ be the set of indices $i$ such that $\nu_i = \lambda_{i+1}-1$ and $M'$ be the set of indices $i$ such that $\nu_i =\lambda_i$. Let $m$ and $m'$ be their respective cardinalities. The generic case is when $m=m'=0$. $\nu$ interlaces $\lambda$ if and only if $m=0$. Note that the weights occurring in the relative Weyl numerator corresponds to the extremal case $m+m'=n$ and $m>0$. We first observe,

\begin{lemma}\label{lemma-epsilonrestricted}
Let $\nu$ satisfy the conclusion of Lemma \ref{lemma-almostinterlace}, i.e. $\lambda_i \geq \nu_i \geq \lambda_{i+1}-1$. Suppose there exists an $\varepsilon' \in \{0,1\}^n$ such that $\nu' = \nu + \varepsilon'$ interlaces $\lambda$. Then 
$$m \leq |\varepsilon'| \leq n- m'.$$
\end{lemma}

\begin{proof}
Suppose $\nu' = \nu + \varepsilon'$, where $\nu' \preceq \lambda$. For $i \in M$, $\nu_i +1 = \lambda_{i+1}$, which forces $\nu'_i =\lambda_{i+1}$ and $\varepsilon'_i =1$. Similarly, if $i \in M'$, then  $\nu_i = \lambda_i $ and hence $\nu'_i =\lambda_{i}$ and $\varepsilon'_i =0$. Thus out of the $|\varepsilon'|$ many 1's, $m$ of them are fixed, and out of the $n - |\varepsilon'|$ many 0's, $m'$ of them are fixed. Hence if $|\varepsilon'|<m$ or $n-|\varepsilon'|<m'$, then there is no possible $\varepsilon'$ satisfying the requirements of the hypothesis. If $m \leq |\varepsilon'| \leq n- m'$, then we can get a $\varepsilon'=\nu'-\nu$ satisfying the hypothesis. 
\end{proof}

We now count the number of ways $\nu$ can be modified by $\varepsilon' \in \{0,1 \}^n$, such that  $\nu'=\nu+\varepsilon'$ interlaces $\lambda$.

\begin{lemma}\label{lemma1}
Given a weight $\nu$ such that $|\nu| +k = |\lambda|$, then the multiplicity of $\bigchi_{\nu + \omega_n}$  in  $\displaystyle \sum_{ \substack{k+ |\nu'+ \varepsilon|=|\lambda|+n\\ \nu' \preceq \lambda}}  \bigchi_{\nu'+\varepsilon}$  is given by 
$$ \binom{n-m'-m}{|\varepsilon'|-m}, $$
if $\lambda_i \geq \nu_i \geq \lambda_{i+1}-1$ and zero otherwise.

\end{lemma}

\begin{proof}
For such a $\nu$, we want to count $\varepsilon' \in \{0,1 \}^n$, such that $\nu' = \nu+ \varepsilon'$ interlaces $\lambda$.

By Lemma \ref{lemma-almostinterlace}, we get that multiplicity of $\bigchi_{\nu}$ is non zero only if  $\lambda_i \geq \nu_i \geq \lambda_{i+1}-1$.

We are now reduced to the case $\lambda_i \geq \nu_i \geq \lambda_{i+1}-1$ for all $i$. 

By Lemma \ref{lemma-epsilonrestricted}, we get that if $|\varepsilon'|<m$ or $n-|\varepsilon'|<m'$, then the multiplicity of $\bigchi_{\nu}$ is zero. As in the generic case, we can now freely choose $|\varepsilon'|-m$ indices amongst  $\{ 1,2, \ldots,n\} \backslash (M \cup M')$ positions. This gives the multiplicity as 
$$\binom{n-m-m'}{ |\varepsilon'|-m} = \binom{n-m-m'}{n- |\varepsilon'|-m'}.$$

\end{proof}

\begin{corollary}\label{corol-multiplicities}
Given $\nu$ with $|\nu| +k = |\lambda|$, and $\nu$ satisfies equation \ref{eq-weaklyinterlace}, the multiplicity $n(\nu)$ of $\bigchi_{\nu+\omega_n}$ in 
$$\sum_{|\varepsilon|<n} (-1)^{n- |\varepsilon|} \left(\sum_{ \substack{|\nu'+\varepsilon|+k=|\lambda|+n\\ \nu' \preceq \lambda}}  \bigchi_{\nu'+\varepsilon} \right)=
\begin{cases}
    \displaystyle -1  & \quad \text{if } m=0\\
    (-1)^{m} & \quad \text{if } m+m' =n,m>0\\
   0  & \quad \text{otherwise.}
  \end{cases}
$$
\end{corollary}

\begin{proof}
We want to count $\varepsilon\in \{0,1 \}^n$ such that $\nu + \omega_n = \nu' + \varepsilon$, which is equivalent to counting $\varepsilon'\in \{0,1 \}^n$ such that $\nu'  = \nu + \varepsilon'$. By Lemmas \ref{lemma-almostinterlace} , \ref{lemma-epsilonrestricted} and \ref{lemma1}, we get that the multiplicity of $\bigchi_{\nu+\omega_n}$ is given by 
$$\sum_{\substack{m \leq |\varepsilon'| \leq n-m'\\ |\varepsilon'|>0}} (-1)^{|\varepsilon'|} \binom{n-m-m'}{|\varepsilon'|-m}. $$

When $m>0$ and $ n=m+m'$, the only term in the above sum is when $|\varepsilon'|=m$.  Hence the sum is equal to $(-1)^m$.

 When $m>0$, upto a sign, the above sum is equal to $(1-1)^{n-m-m'}$ and hence vanishes.

When $m=0$, the above expression differs from $(1-1)^{n-m'}$ by $1$, hence the sum is equal to $-1$.

This covers all the cases, hence completes the proof of the corollary.

\end{proof}

\subsubsection{Proof of Theorem \ref{theorem:branching-gln}}\quad \\We have the relative Weyl numerator. 
$$\bigchi_{\lambda}|_{H} \Delta= t^{\lambda_{n+1}} \bigchi_{\lambda^{(n+1)}}  + \ldots  +(-1)^{n-j+1} \bigchi_{\lambda^{(j)}} + \ldots+ (-1)^{n} t^{\lambda_{1}+n} \bigchi_{\lambda^{(1)}}.$$
We see that 
$$\lambda^{(j)} - \omega_n = (\lambda_1, \ldots, \lambda_{j-1}, \lambda_{j+1}-1,\ldots,\lambda_n-1),$$
satisfies equation (\ref{eq-weaklyinterlace}). Since $n(\nu)$ vanishes if $\nu$ does not satisfy equation (\ref{eq-weaklyinterlace}), it follows that $m(\nu)$ vanishes for such weights.

Suppose $\nu$ is of the form $\lambda^{(j)} - \omega_n$ for some $j$. We have where 
$$m=n-j+1>0, \, m'=j-1 \text{ and } n=m+m'.$$  
The multiplicity with which $\lambda^{(j)}$ occurs in the relative Weyl numerator is $(-1)^{n-j+1}$. As this is equal to $m(\nu) + (-1)^m$, we get that $m(\nu)=0$.

Suppose $\nu$ satisfies equation (\ref{eq-weaklyinterlace}) and $\nu + \omega_n \neq \lambda^{(j)}$ for any $j$. Since $n(\nu)=-1$ precisely when $\nu$ interlaces $\lambda$, it follows that $m(\nu)=1$ if and only if $\nu$ interlaces $\lambda$.

This completes the proof of Theorem \ref{theorem:branching-gln}.

 \begin{rmk}
 \normalfont
It is possible to do the comparison of coefficients from the opposite side, so the induction would work in the decreasing order of the degree of $t$. 
 \end{rmk}

 \begin{rmk}
 \normalfont
At each stage in the proof, we try to match the least degree coefficient of $t$ in an inductive order, knowing the values for lesser powers of $t$. This is essentially a long division of the relative Weyl character formula. 
 \end{rmk}

 \begin{rmk}
 \normalfont
Our initial method was to formally invert the relative Weyl denominator. For any natural number $m$, we have the equation, 
\[ \sum_{i=0}^n (-1)^i\Lambda^i(V)S^{m-i}(V)=0,\]
in the category of virtual representations of $GL(V)$. Here, $\Lambda^i(V)$ (resp. $S^i(V)$ denotes the $i$-th exterior (resp. symmetric) power representation of the standard representation of $GL(V)$ on $V$ and $n$ is the dimension of $V$. Using the above equation, the formal inverse of the Weyl denominator turns out to be, 
\[ t^{-n}\sum_{j=0}^{\infty} t^j S^j(V).\]
This can also be seen directly from the product form of the relative Weyl denominator. Now we appeal to (usual) Pieri formula involving tensor products with symmetric power representation and argue inductively to get at the branching. 
See  \cite{li2018branching} for a proof along these lines. 

However, when considering branching from odd to even orthogonal groups of the same rank (following section), we run into the problem of making sense of inverting the relative Weyl denominator. This led us to consider the more direct approach outlined here to proving the branching formula, which has the advantage that it extends to other cases as well.
 \end{rmk}

\section{Branching rules for $(Spin(2n+1), Spin(2n))$} \label{section_branchingspinodd}

\setcounter{theoremc}{1}
\setcounter{rmkc}{0}
In this Section, we give a proof of Theorem \ref{thmbranchingbn}, giving the classical branching rules for orthogonal groups, with the subgroup $H$ being regular. We follow the schema of proof as given for $(GL(n+1),GL(n))$: give a $(G,H)$-relative Weyl character formula, a $(G,H)$-relative Pieri formula and finally the branching rule is derived in an inductive manner following the process of long division. 

Although the relative Weyl character formula looks different from that of $GL(n)$, the combinatorics of the relative Pieri formula is similar to that of the $GL(n)$ case. Hence, the derivation of the branching rule from that of the relative Pieri is similar to that of the general linear case.

\subsection{Weyl character formula.} We follow the notation and convention as in \cite[Chapter 24]{fulton2013representation}.
For $m\geq 4$, let $Spin(m)$ be the simply connected double cover of $SO(m, \mathbb{C})$.

We consider $\mathbb{R}^n$ with the standard basis $L_i$. Denote by $x_i^{\pm 1} = e^{\pm L_i}$  the formal exponents of $\pm L_i$.
Given a tuple $\eta= (\eta_1, \eta_2  \ldots, \eta_{n}) = \sum_i \eta_i L_i$,  define
$$D^+(\eta)= \det |x_j^{\eta_i} + x_j^{-\eta_i}| \quad\text{ and } \quad D^-(\eta)= \det |x_j^{\eta_i} - x_j^{-\eta_i}|.$$
We observe that $D^+(\eta)$ is invariant under sign changes of $\eta$, whereas $D^-(\eta)$ is alternating with respect to sign changes.
 
\subsubsection{$Spin(2n+1)$.} Let $G=Spin(2n+1)$ and $H= Spin(2n)$.
The set of positive roots of $Spin(2n+1)$ is given by:
$$\Phi^+(B_n) = \{L_i \pm L_j \mid 1\leq i<j \leq n\} \cup \{L_i \mid 1 \leq i \leq n\}.$$
The highest weights of $G$ are described by $n-$tuples $$\lambda= (\lambda_1, \lambda_2  \ldots, \lambda_{n}),$$  
where $\lambda_1 \geq \lambda_2 \geq \ldots \geq \lambda_{n} \geq 0$ are all integers or all half integers. 
Let  $$\rho_{G} = (n-1+1/2, n-2+1/2, \ldots, 1/2),$$ be half the sum of positive roots for $G$. The character for the highest weight representation with highest weight $\lambda$ is given by,  
\begin{equation}\label{weylchformulabn}
\bigchi_{\lambda} = \frac{D^-(\lambda + \rho_G)}{D^-(\rho_G)}. 
\end{equation}

\subsubsection{$Spin(2n)$} The set of positive roots of $Spin(2n)$ is given by:
$$\Phi^+(D_n) = \{L_i \pm L_j | 1\leq i<j \leq n\} .$$
The highest weights of $H$  are given by an $n-$tuple 
$$\mu = (\mu_1,\mu_2,\ldots, \mu_n),$$   
where $\mu_1 \geq \mu_2 \geq \ldots \geq |\mu_{n}| \geq 0$ are all integers or all half integers. Let  
$$\rho_{H} = \sum (n -i )L_i= (n-1,n-2, \ldots , 0)$$ 
be half the sum of positive roots of $H$. The Weyl character formula for the irreducible representation of $H$ with highest weight $\mu$ is given by, 
\begin{equation}\label{weylchformuladn}
\bigchi_{\mu} = \frac{D^-(\mu + \rho_H) + D^+(\mu + \rho_H)}{D^+(\rho_H)}.
\end{equation}

\subsection{Relative Weyl character formula}\quad\\
The numerator and denominator of the  Weyl character formula for $G$ involve a sum of weights of $G$. 
Like in the case of $GL(n)$, we now get a relative Weyl character formula for $G$ with respect to $H$, where the numerator and denominator are sums of characters of $H$. As above, we identify the tori of $g$ and $H$.

\begin{proposition}[Relative Weyl character formula]\label{relativeweylchbn}
Given $\lambda= (\lambda_1, \lambda_2  \ldots, \lambda_{n})$ is a highest weight for $G$, then 
$$\bigchi_\lambda|_H = \frac{\bigchi_{\lambda^+}- \bigchi_{\lambda^-}}{\Delta}, $$
\begin{gather*}
\text{where }\,\lambda^+ = \left(\lambda_1+\frac{1}{2}, \lambda_2 +\frac{1}{2}, \ldots , \lambda_n +\frac{1}{2}\right), \\ 
\lambda^- = \left(\lambda_1+\frac{1}{2}, \lambda_2 +\frac{1}{2}, \ldots , -\lambda_n -\frac{1}{2}\right) , 
\end{gather*}
are highest weights for $H$.  
Here,
$$\Delta = \prod_{i=1}^n (x_i^{1/2} - x_i^{-1/2}) = S^+-S^-,$$ is the relative Weyl denominator for the pair $(G,H)$, where $S^{\pm}$ are the half spin representations of $H$ corresponding to the highest weights $\lambda^{\pm}$ with $\lambda= (0,\ldots,0).$ 
\end{proposition}

\begin{proof}
By the product formula for the Weyl denominator, we observe that 
\begin{align*}
D^-(\rho_G) &= \prod_{\alpha \in \Phi^+(B_n)}(e^{\alpha/2}-e^{-\alpha/2})\\
&=\prod_{\alpha \in \Phi^+(D_n)}(e^{\alpha/2}-e^{-\alpha/2}) \prod_{i=1}^n (x_i^{1/2} - x_i^{-1/2}) \\
&= \frac{D^+(\rho_H)}{2} \prod_{i=1}^n (x_i^{1/2} - x_i^{-1/2}) ,
\end{align*}
where $e^{\pm \alpha/2}$ is the formal exponent corresponding to half the root $\pm \alpha$.
Note that $\lambda^+$ and $ \lambda^-$ are equal in all but the last entry, where they are negative of each other.  As $D^-$ is alternating under sign changes in $\eta$, we have that
$$D^-(\lambda^+ + \rho_H) = -D^-(\lambda^- + \rho_H),  $$
whereas $D^+$ is invariant under sign changes, giving us
$$D^+(\lambda^+ + \rho_H) = D^+(\lambda^- + \rho_H).$$
We also have that $\lambda+ \rho_G=\lambda^+ + \rho_H$, which gives us
$$D^-(\lambda+ \rho_G) = D^-(\lambda^+ + \rho_H). $$
Using the above identities, we can express the numerator  $D^-(\lambda +\rho_G)$ of $\bigchi_{\lambda}$ as, 
$$ \frac{\left( D^-(\lambda^+ + \rho_H) + D^+(\lambda^+ + \rho_H) \right) - \left(D^-(\lambda^- + \rho_H) + D^+(\lambda^- + \rho_H)\right)}{2}. $$

Dividing the Weyl numerator with the Weyl denominator gives us the relative Weyl character.
The denominator corresponds to taking $\lambda = (0,\ldots,0)$. This gives us the half spin representations.
\end{proof}

\begin{rmk}
\normalfont
One can arrive at the above Proposition by looking at the expansion of the Weyl character formula with respect to the the cosets of the Weyl group of $H$ in the Weyl group of $G$, which has been expounded in \cite{gross1998weyl}.
\end{rmk}

\subsection{Relative Pieri formula}\quad
 We now derive the relative Pieri formula for $(G,H)$, by which we mean the tensor product decomposition of a highest weight representation of $H$ with that of the relative Weyl denominator $\Delta$. One can prove the relative Pieri formula by proving the tensor product decompositions with $S^{\pm}$ individually. However, here we follow the approach of Okada's \cite{Okada16}, where we use the determinantal identities for the Weyl character formula along with the product expansion for the relative Weyl denominator.

\begin{proposition}[Relative Pieri formula]\label{relpierispin}
Given $\mu$ a dominant weight of $Spin(2n)$, we have the following decomposition:
$$\bigchi_{\mu}  \Delta  =\sum_{\substack{\varepsilon\in \{\pm 1\}^n\\\mu + \varepsilon/2 \text{ dominant}  }} (-1)^{\varepsilon} \bigchi_{\mu + \varepsilon/2}.$$
    
\end{proposition}

\begin{proof}
For the dominant weight $\mu$ of $H$, by the Weyl character formula for $H$ (equation (\ref{weylchformuladn}))
$$\bigchi_{\mu} \Delta= \frac{\Delta D^-(\mu + \rho_H) + \Delta D^+(\mu + \rho_H)}{D^+(\rho_H)}.$$
We simplify the numerator in the above expression. We first observe that, 
$$(x_j^{l_i} - x_j^{-l_i})(x_j^{1/2} - x_j^{-1/2}) = (x_j^{l_i+1/2} + x_j^{-(l_i+1/2)}) - (x_j^{l_i-1/2} + x_j^{-(l_i-1/2)}), $$
and 
$$(x_j^{l_i} + x_j^{-l_i})(x_j^{1/2} - x_j^{-1/2}) = (x_j^{l_i+1/2} - x_j^{-(l_i+1/2)}) - (x_j^{l_i-1/2} - x_j^{-(l_i-1/2)}), $$
where $l=(l_1, l_2, \ldots, l_n)$ is a tuple of all distinct integers or all half integers. Using the $n$-linearlity of the determinant as in the case of $GL(n)$, along with the above identities  gives us:

\begin{align}
\Delta D^-(l)  &= \prod_{k=1}^n (x_k^{1/2} - x_k^{-1/2}) \left(\det|x_i^{l_j} - x_i^{-l_j}|  \right)\nonumber\\
&= \det|(x_i^{l_j} - x_i^{-l_j})(x_i^{1/2} - x_i^{-1/2})|\nonumber \\
&= \det|(x_i^{l_j+1/2} + x_i^{-(l_j+1/2)}) - (x_i^{l_j-1/2} + x_i^{-(l_j-1/2)})|\nonumber\\
&=\sum_{\varepsilon\in \{\pm 1\}^n} (-1)^{\varepsilon} \det|x_i^{l_j+\varepsilon_j/2} + x_i^{-(l_j+\varepsilon_j/2)}|\nonumber\\
&=\sum_{\varepsilon\in \{\pm 1\}^n} (-1)^{\varepsilon} D^+(l+\varepsilon/2)\label{relpierid-},
\end{align}
where $(-1)^{\varepsilon} = \prod_{i=1}^n \varepsilon_i.$
With similar calculation, we get
\begin{equation}\label{relpierid+}
\Delta  D^+(l) =   \sum_{\varepsilon\in \{\pm 1\}^n} (-1)^{\varepsilon} D^-(l +\varepsilon/2).
\end{equation}
For $l = \mu + \rho_H$ , adding (\ref{relpierid+}) and (\ref{relpierid-}), and dividing by the Weyl denominator, we get
 \begin{align}
\bigchi_{\mu} \Delta &= \frac{\Delta D^-(\mu + \rho_H) + \Delta D^+(\mu + \rho_H)}{D^+(\rho_H)}\nonumber\\ 
&=\sum_{\varepsilon\in \{\pm 1\}^n} (-1)^{\varepsilon} \frac{D^-(\mu +\varepsilon/2+ \rho_H) +  D^+(\mu +\varepsilon/2+ \rho_H)}{D^+(\rho_H)}\nonumber\\
&= \sum_{\varepsilon\in \{\pm 1\}^n} (-1)^{\varepsilon}  \bigchi_{\mu + \varepsilon/2}.
 \end{align}
Note that if there is some $\varepsilon$ such that $\mu + \varepsilon/2$ is not dominant, then there exists a $j$ such that $\mu_j + n -j +\varepsilon_j = \mu_{j+1} + n - j -1 + \varepsilon_{j+1}$. This gives us that the $j^{th}$ and $(j+1)^{th}$ rows are identical in both the determinants, therefore the numerator vanishes. Hence the non-zero summands corresponds to those $\varepsilon$ for which $\mu + \varepsilon/2$ which are dominant. This completes the proof.
\end{proof}

\begin{rmk}
\normalfont
An alternate way of proving the relative Pieri formula is to consider the tensor product with the half spin representations $S^{\pm}$ individually, using  the fact that $S^{\pm}$ are both miniscule representations of $H$. Then using the theorem for tensor product for miniscule representations (\cite[Corollary 3.5]{shrawantensor} , we get the desired result. We give now a different derivation using the determinental expressions.
\end{rmk}

\begin{corollary}
 Let $\lambda$ be a dominant weight for $H$ and $S^{\pm}$ are its half spin representations. Then

 \begin{align}
 \bigchi_{\mu}  S^+ &= \sum_{\substack{\varepsilon\in \{\pm 1\}^n, |\varepsilon| \text{ even}\\ \mu + \varepsilon/2 \text{ dominant}  }} \bigchi_{\mu + \varepsilon/2}, \\ \bigchi_{\mu}  S^+&= \sum_{\substack{\varepsilon\in \{\pm 1\}^n, |\varepsilon| \text{ odd}\\ \mu + \varepsilon/2 \text{ dominant}  }} \bigchi_{\mu + \varepsilon/2}.
 \end{align}

\end{corollary}

\begin{proof}
Note that $S^+$ and $S^-$ are irreducible representations of $H$ corresponding to the weights $(\tfrac{1}{2},\tfrac{1}{2},\ldots,\tfrac{1}{2},\tfrac{1}{2})$ and $(\tfrac{1}{2},\tfrac{1}{2},\ldots,\tfrac{1}{2},-\tfrac{1}{2})$ respectively. Also $\Delta = S^+-S^-$ is a virtual character of $H$. Let the character $\Delta' = \prod_{i=1}^n (x_i^{1/2} + x_i^{-1/2})$, and we can see that $\Delta' = S^++S^-$. Using the method as in  proposition \ref{relpierispin}, one can prove its tensor product decomposition with any highest weight representation. Combining this with the result from the proposition gives a proof of the corollary. 
\end{proof}

\subsection{Branching formula.} We recall the notion of interlacing for the pair $Spin(2n+1),Spin(2n)$. Suppose $\lambda= (\lambda_1, \lambda_2  \ldots, \lambda_{n})$  (resp. $\nu = (\nu_1,\nu_2, \ldots, \nu_n)$) is a dominant weight of $Spin(2n+1)$ (resp. $Spin(2n)$). We say  that $\nu$ interlaces $\lambda$, denoted by $\nu \preceq\lambda$,  if $\lambda_1 \geq \nu_1 \geq \lambda_2 \geq \ldots \geq \nu_{n-1} \geq \lambda_n \geq |\nu_n|$.

We will now prove the following branching rule:
\begin{theorem}
   Suppose $\lambda$ is a dominant weight of $Spin(2n+1)$.
     \begin{equation}\label{eq_multiplicityspin}
     \bigchi_{\lambda}|_{H} = \sum_{\nu} m(\lambda,\nu) \bigchi_{\nu}, 
    \end{equation}
    where the sum is over the dominant weights $\nu$  of $Spin(2n)$. 
    
    Then $m(\lambda,\nu) =1$ if and only if $\nu \preceq \lambda$ , and zero otherwise.
\end{theorem}

The proof will proceed along the same lines as for $GL(n)$. 
    In the following since we are working with a fixed $\lambda$, we will use $m(\nu)$ instead of $m(\lambda,\nu)$.

    We will prove the theorem via ( decreasing ) induction on the weights with respect to the lexicographic ordering ($>_l$).  Multiplying the relative Weyl character formula (Proposition \ref{relativeweylchbn}) by the relative Weyl denominator, we get:
\begin{align}
     \bigchi_{\lambda^+}- \bigchi_{\lambda^-} &= \bigchi_{\lambda}|_{H}  \Delta \nonumber \\
     &=  \sum_{\nu} m(\nu) (\bigchi_{\nu}  \Delta) \nonumber \\
     &= \sum_{\nu} m(\nu) \left(\sum_{\varepsilon\in \{\pm 1\}^n} (-1)^{\varepsilon} \bigchi_{\nu + \varepsilon/2}\right),\label{eq-relweylpierispin}
    \end{align}
where we use the relative Pieri formula (Proposition \ref{relpierispin} ) in the last equality.

    Suppose $\mu$ is the largest in the lexicographic ordering  such that $m(\mu) >0$. By the relative Pieri formula, $\bigchi_{\mu} \Delta $ has $\bigchi_{\mu + (\frac{1}{2},\frac{1}{2},\ldots,\frac{1}{2})}$ as sub-representation for which the highest weight is largest in lexicographic ordering. Comparing it in equation (\ref{eq-relweylpierispin}) with the largest in lexicographic ordering gives us $\mu = \lambda^+ - (\tfrac{1}{2},\tfrac{1}{2},\ldots,\tfrac{1}{2}) =\lambda$ and $m(\mu) = m(\lambda)=1$. 

    Let $\nu <_l \lambda$ be a dominant weight for $H$ . For our induction hypothesis, we assume that all weights  $\nu'$ such that $\nu' >_l \nu$ satisfy the hypothesis in the theorem. We can rewrite equation (\ref{eq_multiplicityspin}) as:

\begin{equation}\label{eq_inductionspin}
    \bigchi_{\lambda}|_{H} = \sum_{\substack{\nu' >_l \nu \\ \nu \preceq \lambda}} \bigchi_{\nu'} +m(\nu) \bigchi_{\nu} +\sum_{\mu <_l \nu}m(\mu) \bigchi_{\mu} .
\end{equation}

We would like to understand the multiplicity $m(\nu)$ for the highest weight $\nu$. We now multiply the above equation by $\Delta$ and use relative Pieri formula. The middle term in the right hand side contributes a term $m(\nu)\bigchi_{\nu + (\frac{1}{2},\frac{1}{2},\ldots,\frac{1}{2})}$. 
If $\mu <_l \nu$, then there does not exist any $\varepsilon\in \{\pm 1\}^n$ such that $\mu +\varepsilon/2 = \nu + (\frac{1}{2},\frac{1}{2},\ldots,\frac{1}{2})$. This is equivalent to saying that the third sum on the right side does not contribute to $\bigchi_{\nu + (\frac{1}{2},\frac{1}{2},\ldots,\frac{1}{2})}$. Hence the only possible contribution to $\bigchi_{\nu + (\frac{1}{2},\frac{1}{2},\ldots,\frac{1}{2})}$ comes from the first two terms. Multiplying the first term on the right side in the above equation by $\Delta$ and using relative Pieri formula gives us
\begin{equation}
 \sum_{\substack{\nu' >_l \nu \\ \nu \preceq \lambda}} \bigchi_{\nu'} \Delta = \sum_{\substack{\nu' >_l \nu \\ \nu \preceq \lambda}} \left(\sum_{\substack{\varepsilon\in \{\pm 1\}^n \\ \nu' + \varepsilon/2 \text{ dominant}}} (-1)^{\varepsilon} \bigchi_{\nu' + \varepsilon/2} \right).
\end{equation}

To understand the multiplicity $m(\nu)$, we need to know the multiplicity say $n(\nu)$ of  $\bigchi_{\nu + (\frac{1}{2},\frac{1}{2},\ldots,\frac{1}{2})}$ in the above expression. To do this, we need to understand when $\nu' + \varepsilon/2 = \nu + (\frac{1}{2},\frac{1}{2},\ldots,\frac{1}{2})$ in the above expression. Note that this can be seen as $\nu' = \nu + \varepsilon'$, where $\varepsilon' = (\frac{1}{2},\frac{1}{2},\ldots,\frac{1}{2}) -\varepsilon \in \{ 0,1\}^n$. One can see that $\varepsilon'_i =0$ if $\varepsilon_i=1$, and $\varepsilon'_i =1$ if $\varepsilon_i=-1$, so we have that $(-1)^{\varepsilon'}= (-1)^{\varepsilon}$ and $|\varepsilon'|>0$ ( as $\nu' >_l\nu$). In this section , for any $\varepsilon \in \{ \pm 1\}^n$, define $\varepsilon' = (\frac{1}{2},\frac{1}{2},\ldots,\frac{1}{2}) - \varepsilon/2$. The rest of the proof is similar to the case of $GL(n)$.

We first show that $\nu$ weakly interlaces $\lambda$ in the following sense:

\begin{lemma}\label{lemma_spinalmostinterlace}
Let $\nu$ be such that $\nu'  = \nu + \varepsilon'$ interlaces $\lambda$ for some  $\varepsilon' \in \{0,1\}^n,$ with $|\varepsilon'|>0.$  Then  

\begin{equation} \label{eq_spinweaklyinterlace}
\lambda_i \geq \nu_i \geq \lambda_{i+1}-1 \text{ if } 1\leq i \leq n-1 \, \text{ and } \lambda_n \geq \nu_n \geq -\lambda_n-1.
\end{equation}.
\end{lemma}

\begin{proof}
Given that $\nu'_i=\nu_i + \varepsilon'_i $, and as $\nu' \preceq \lambda$, we get that  $\lambda_i \geq \nu'_i \geq \lambda_{i+1}$ for $1 \leq i \leq n-1$ and $\lambda_n \geq |\nu'_n|$. Putting the two of them together, we get that $\lambda_i \geq \nu_i \geq \lambda_{i+1}-1$  if $ 1\leq i \leq n-1$ and $ \lambda_n \geq \nu_n \geq -\lambda_n-1$.
\end{proof}

Define $M =\{ i \mid \nu_i = \lambda_{i+1}-1 \text { for } 1 \leq i \leq n-1\} \cup \{n \mid \nu_n = -\lambda_n-1 \}$ and
$M' = \{ i \mid \nu_i =\lambda_i\}$. Let $m = |M|$ and $m'=|M'|$. We first observe

\begin{lemma}\label{lemma_spinepsilonrestricted}
Let $\nu$ satisfy the conclusion of Lemma \ref{lemma_spinalmostinterlace}, i.e. $\lambda_i \geq \nu_i \geq \lambda_{i+1}-1$ for $1 \leq i \leq n-1$ and $\lambda_n \geq \nu_n \geq -\lambda_n-1$. Suppose there exists an $\varepsilon' \in \{0,1\}^n$ such that $\nu' = \nu + \varepsilon'$ interlaces $\lambda$. Then 
$$m \leq |\varepsilon'| \leq n- m'.$$
\end{lemma}

\begin{proof}
Given $\nu' = \nu + \varepsilon'$ where $\nu' \preceq \lambda$. For $i \in M$, $\nu_i +1 = \lambda_{i+1}$, which forces $\nu'_i =\lambda_{i+1}$ and $\varepsilon'_i =1$. Similarly, if $i \in M'$, we have that $\nu_i = \lambda_i $, which forces $\nu'_i =\lambda_{i}$ and $\varepsilon'_i =0$. Hence out of the $|\varepsilon'|$ many 1's, $m$ of them are fixed, and out of the $n - |\varepsilon'|$ many 0's, $m'$ of them are fixed. Hence if $|\varepsilon'|<m$ or $n-|\varepsilon'|<m'$, then there is no possible $\varepsilon'$ satisfying the requirements of the hypothesis. If $m \leq |\varepsilon'| \leq n- m'$, then we can get a $\varepsilon'=\nu'-\nu$ satisfying the hypothesis. 
\end{proof}

We now count the number of ways $\nu$ can be modified by $\varepsilon' \in \{0,1 \}^n$, such that  $\nu'=\nu+\varepsilon'$ interlaces $\lambda$.

\begin{lemma}\label{lemma_spinmultiplicitycount}
Given a weight $\nu$ and a fixed $k>0$, the multiplicity of $\bigchi_{\nu + (\frac{1}{2},\frac{1}{2},\ldots,\frac{1}{2})}$  in  $\displaystyle \sum_{ \substack{\varepsilon \in \{ \pm 1\}^n, |\varepsilon'|=k\\ \nu' \preceq \lambda}}  \bigchi_{\nu'+\varepsilon/2}$  is given by 
$$ \binom{n-m'-m}{|\varepsilon'|-m}, $$
if $\nu$ weakly interlaces $\lambda$ (equation(\ref{eq_spinweaklyinterlace})) and zero otherwise, where ${\binom{l}{s}}$ denotes the standard binomial coefficient.

\end{lemma}

\begin{proof}
For such a $\nu$, we want to count how many $\varepsilon' \in \{0,1 \}^n$ such that $\nu' = \nu+ \varepsilon'$ interlaces $\lambda$.

By Lemma \ref{lemma_spinalmostinterlace}, we get that multiplicity of $\bigchi_{\nu}$ is non zero only if  $\lambda_i \geq \nu_i \geq \lambda_{i+1}-1$ for $1 \leq i \leq n-1$ and $\lambda_n \geq \nu_n \geq -\lambda_n-1$.

We are now reduced to the case $\lambda_i \geq \nu_i \geq \lambda_{i+1}-1$ for $1 \leq i \leq n-1$ and $\lambda_n \geq \nu_n \geq -\lambda_n-1$.

By Lemma \ref{lemma_spinepsilonrestricted}, we get that if $|\varepsilon'|<m$ or $n-|\varepsilon'|<m'$, then the multiplicity of $\bigchi_{\nu+ (\frac{1}{2},\frac{1}{2},\ldots,\frac{1}{2})}$ is zero. As in the generic case, we can now freely choose $|\varepsilon'|-m$ indices amongst  $\{ 1,2, \ldots,n\} \backslash (M \cup M')$ positions. This gives the multiplicity as 
$${\binom{n-m-m'}{ |\varepsilon'|-m}} = {\binom{n-m-m'}{n- |\varepsilon'|-m'}}.$$

\end{proof}

\begin{corollary}\label{corol_spinmultiplicities}
Given $\nu$  that weakly interlaces $\lambda$ (equation (\ref{eq_spinweaklyinterlace})), the multiplicity $n(\nu)$ of $\bigchi_{\nu+(\frac{1}{2},\frac{1}{2},\ldots,\frac{1}{2})}$ in 
$$\sum_{k=1}^n (-1)^{|\varepsilon'|} \left(\displaystyle \sum_{ \substack{\varepsilon \in \{ \pm 1\}^n, |\varepsilon'|=k\\ \nu' \preceq \lambda}}  \bigchi_{\nu'+\varepsilon/2} \right)=
\begin{cases}
    \displaystyle -1  & \quad \text{if } m=0\\
    (-1)^{m} & \quad \text{if } m+m' =n,m>0\\
   0  & \quad \text{otherwise.}
  \end{cases}
$$
\end{corollary}

\begin{proof}
We want to count $\varepsilon\in \{0,1 \}^n$ such that $\nu + (\frac{1}{2},\frac{1}{2},\ldots,\frac{1}{2}) = \nu' + \varepsilon/2$, which is equivalent to counting $\varepsilon'\in \{0,1 \}^n$ such that $\nu'  = \nu + \varepsilon'$. By Lemmas \ref{lemma_spinalmostinterlace} , \ref{lemma_spinepsilonrestricted} and \ref{lemma_spinmultiplicitycount}, we get that the multiplicity of $\bigchi_{\nu+(\frac{1}{2},\frac{1}{2},\ldots,\frac{1}{2})}$ is given by 
$$\sum_{\substack{m \leq |\varepsilon'| \leq n-m'\\ |\varepsilon'|>0}} (-1)^{|\varepsilon'|} {\binom{n-m-m'}{|\varepsilon'|-m}}. $$

When $m>0$ and $ n=m+m'$, the only term in the above sum is when $|\varepsilon'|=m$.  Hence the sum is equal to $(-1)^m$.

 When $m>0$, upto a sign, the above sum is equal to $(1-1)^{n-m-m'}$ and hence vanishes.

When $m=0$, the above expression differs from $(1-1)^{n-m'}$ by $1$, hence the sum is equal to $-1$.

This covers all the cases, hence proves the corollary.

\end{proof}

\subsubsection{Proof of Theorem \ref{thmbranchingbn}}\quad \\We have the relative Weyl numerator. 
$$\bigchi_{\lambda}|_{H} \Delta= \bigchi_{\lambda^+} - \bigchi_{\lambda^-}.$$
We see that 
$$\lambda^- - (\tfrac{1}{2},\tfrac{1}{2},\ldots,\tfrac{1}{2}) = (\lambda_1, \ldots,\lambda_{n-1}, -\lambda_n-1)$$
satisfies equation (\ref{eq_spinweaklyinterlace}). Since $n(\nu)$ vanishes if $\nu$ does not satisfy equation (\ref{eq_spinweaklyinterlace}), it follows that $m(\nu)$ vanishes for such weights.

Suppose $\nu = \lambda^- - (\tfrac{1}{2},\tfrac{1}{2},\ldots,\tfrac{1}{2})$ , then we get that $m=1$ and $m'=n-1$, hence $n(\nu)=-1$. Also the multiplicity with which $\bigchi_{\lambda^-}$ occurs in the Weyl numerator is $-1$. As this is equal to $m(\nu) + (-1)$, we get that $m(\nu)=0$.


Suppose $\nu$ satisfies equation (\ref{eq_spinweaklyinterlace}) and $\nu + (\tfrac{1}{2},\tfrac{1}{2},\ldots,\tfrac{1}{2}) \neq \lambda^-$. Since $n(\nu)=-1$ precisely when $\nu$ interlaces $\lambda$, it follows that $m(\nu)=1$ if and only if $\nu$ interlaces $\lambda$.

This completes the proof of Theorem \ref{thmbranchingbn}.

\section{Branching rules for $(Sp(2n) , Sp(2)\times Sp(2n-2))$}

In this section, we will give a proof of Theorem \ref{thmbranchingcn}, the `classical' branching rule for symplectic groups from $G=Sp(2n)$ to $H=Sp(2n-2)\times Sp (2)$. Unlike the previous two cases, where the restrictions were multiplicity free, in the case of symplectic the  branching multiplicities are more difficult to get at.  We follow the approach of Yacobi and Wallach (\cite{wallach2009multiplicity}), describing the branching multiplicities in terms of  dimensions of $SL(2)=Sp(2)$ representations. 

Our method is similar in outline as in the linear and orthogonal cases. We first get a relative Weyl character formula. We then derive a relative Pieri formula, where we decompose the tensor of a representation of $H$ with the relative Weyl denominator. However in the proof of branching, we no longer have simple identities like $t^kt^l=t^{k+l}$, which we used while carrying out long division in the linear case, where $t$ is the polynomial parameter coming from $GL(1)$.  In the symplectic case, if we consider the coefficients as representations of $SL(2)$, we do not have simple equations like $S^{(k)}S^{(l)}=S^{(k+l)}$, where $S^{(k)}$ denotes the character of the $k^{th}$ symmetric power of the standard representation. In place of the simple identities, we prove some cancellative identities involving representations of $SL(2)$, which allows us to derive the branching multiplicities from the relative Pieri.


As in the previous section, our notations follow the convention of  \cite[Chapter 24]{fulton2013representation}.  We would work with  $(G,H) =(Sp(2n), Sp(2)\times Sp(2n-2))$, where one can see that $H$ is a regular subgroup of $G$.

\subsection{Relative Weyl character formula} We recall the Weyl character formula for $Sp(2n)$:

The dominant weights for $G$ are represented by $\lambda = (\lambda_1,\lambda_2, \ldots, \lambda_n),$ 
 where $\lambda_1 \geq \lambda_2 \geq \cdots \geq \lambda_n \geq 0$ and  $\lambda_i$ are all integers. The dominant weights for $H$ are given by a pair $(k, \eta)$, where $k$ is a non-negative integer and $\eta$ is a dominant weight for $Sp(2n-2)$. Denote by $S^{(k)} \bigchi_{\eta}$ the character of the irreducible highest weight representation corresponding to the highest weight $(k, \eta)$. This corresponds to the outer tensor product of the symmetric $k^{th}$ power of the standard representation of $Sp(2)$ with the highest weight representation of $Sp(2n-2)$ corresponding to $\eta$.

 The character for the irreducible representation corresponding to the highest weight  $\lambda = (\lambda_1,\lambda_2,\ldots, \lambda_n)$ is given by, 

\begin{equation}\label{eq:weylchformulasymplectic}
\bigchi_{\lambda} = \frac{\det|x_j^{\lambda_i +n-i+1} - x_j^{-(\lambda_i +n-i+1)}|}{\det|x_j^{n-i+1} - x_j^{-(n-i+1)}|}.
\end{equation}

\begin{proposition}{(Relative Weyl character formula)}
Let $\lambda = (\lambda_1, \lambda_2, \ldots, \lambda_{n})$ be a dominant weight for $G=Sp(2n)$ and $H= Sp(2)\times Sp(2n-2))$. Let $\bigchi_{\lambda}$ be the irreducible highest weight representation of $G$ with highest weight $\lambda$. We have the following expression for restriction of $\bigchi_{\lambda}$:
\begin{equation}\label{eq:relweylchformulagsympl}
\bigchi_{\lambda}|_{H}= \frac{\displaystyle\sum_{i=0}^{n-1} (-1)^{i} S^{(\lambda_{n-i}+i)}\bigchi_{\lambda^{(n-i)}}}{\displaystyle\sum_{i=0}^{n-1} (-1)^{i} S^{(i)}\bigchi_{\omega_{n-i-1}}}.
\end{equation}
where 
\begin{itemize}
    \item $S^{(k)}$ is the symmetric $k^{th}$ power of the standard representation of $Sp(2)$,
    \item $\omega_k=(1,1,1,\ldots,1,0,\ldots ,0)$ ($k$ many 1's) corresponds to the $k^{th}$ fundamental weight for $Sp(2n-2)$, 
    \item $\lambda^{(i)} = (\lambda_1+1,\lambda_2+1, \ldots, \lambda_{i-1}+1, \lambda_{i+1}, \ldots, \lambda_{n})$ are dominant weights of $Sp(2n-2)$.
\end{itemize}
\end{proposition}

\begin{proof}
We do a co-factor expansion along the first column from the bottom in the numerator and denominator of (\ref{eq:weylchformulasymplectic}). We consider $x_1$-variable for $Sp(2)$ which gives us: 
\[
\begin{split}
    \bigchi_{\lambda}|_{H}& =  \frac{\displaystyle\sum_{i=0}^{n-1} (-1)^{i} (x_1^{\lambda_{n-1}+i+1} - x_1^{-(\lambda_{n-1}+i+1)})\bigchi_{\lambda^{(n-i)}}}{\displaystyle\sum_{i=0}^{n-1} (-1)^{i} (x_1^{i+1} - x_1^{-(i+1)})\bigchi_{\omega_{n-i-1}}}\\
    &=   \frac{\displaystyle\sum_{i=0}^{n-1} (-1)^{i} S^{(\lambda_{n-i}+i)}\bigchi_{\lambda^{(n-i)}}}{\displaystyle\sum_{i=0}^{n-1} (-1)^{i} S^{(i)}\bigchi_{\omega_{n-i-1}}}.\end{split}
\]
We arrive at the second equality upon dividing both the numerator and denominator by the Weyl denominator of $H= Sp(2)\times Sp(2n-2))$. We note that the relative Weyl denominator is given by
$$\Delta= \displaystyle\sum_{i=0}^{n-1} (-1)^{i} S^{(i)}\bigchi_{\omega_{n-i-1}} = (-x_1)^{-(n-1)}\prod_{i=2}^{n} \left(x_1 - x_i\right)\left(x_1 - x_i^{-1}\right). $$

\end{proof}

\subsection{Relative Pieri formula}
We now derive a relative Pieri formula for $(G,H)$ where $G=Sp(2n)$ and $H=Sp(2)\times Sp(2n-2)$. The proof is a modification of the Okada's method (\cite{Okada16}), where he considers the tensor product of a representation of $Sp(2n)$ with the formal product $\prod_{i=2}^{n} \left(x_1 + x_i\right)\left(x_1 + x_i^{-1}\right)$. 

\begin{theorem}[Relative Pieri formula]
Let $S^{(k)} \bigchi_{\eta}$ be the character of  $H$ corresponding  to a highest weight $(k, \eta)$, of $H$. Then,
$$\Delta \times \left( S^{(k)} \bigchi_{\eta} \right)= S^{(k)}\sum_{\substack{\nu \text{ dominant} \\ \nu-\eta \in \{ \pm 1, 0\}^{n-1}}}  (-S^{(1)})^{n-1- |\nu-\eta|} \bigchi_{\nu}.$$
\end{theorem}

\begin{proof}
We have the Weyl character formula, 
$$\bigchi_{\eta} = \frac{\det|x_i^{\eta_j +n-j}- x_i^{-(\eta_j +n-j)}|}{\det|x_i^{n-j} - x_i^{-(n-j)}|}.$$
Let $t=x_1$. By the product expansion for the Weyl denominator,  
$$\Delta = (-t)^{-(n-1)}\prod_{i=2}^n (t-x_i)(t-x_i^{-1}).$$
In order to make the formulas concise, let $[x_j]^a := x_j^a -x_j^{-a}$. We have the following identity:
$$(t-x_j)(t-x_j^{-1})[x_j]^a= (t^2+1)[x_j]^a - t([x_j]^{a+1} + [x_j]^{a-1}).$$
Let $D_H$ denote the Weyl denominator for $H$. Then,
\begin{align}
\Delta \times \left( S^{(k)} \bigchi_{\eta} \right) &= \frac{S^{(k)}}{D_H(-t)^{n-1}}  \left(\prod_{i=2}^n (t-x_i)(t-x_i^{-1})\right) \det|[x_i]^{\eta_j +n-j+1}| \nonumber\\
&= \frac{S^{(k)}}{D_H(-t)^{n-1}}  \det|(t^2+1)[x_i]^{\eta_j +n-j+1} - t ([x_i]^{\eta_j +n-j+2} + [x_i]^{\eta_j +n-j})|\nonumber\\ 
&= \frac{S^{(k)}}{D_H} \sum_{\varepsilon,\delta \in \{ 0,1\}^{n-1}} (-t)^{|\varepsilon|+|\delta|-n+1} \det|[x_i]^{\eta_j +n-j +1 + \varepsilon_j-\delta_j}| ,\label{eq:symplecticweylnumpieri}
\end{align}
where
\begin{itemize}
    \item $\varepsilon = (\varepsilon_2,\varepsilon_3,\ldots,\varepsilon_n),$
    \item $\delta = (\delta_2,\ldots,\delta_n),$
    \item $|\varepsilon|= \sum_{i=2}^n| \varepsilon_i|$ and $|\delta | =\sum_{i=2}^n |\delta_i|.$
\end{itemize} 
We will show that the summands are non zero only if $\eta +\varepsilon-\delta$ is dominant (call this collection of $(\varepsilon,\delta)$ as $\mathcal{A}$). 

If $\eta + \varepsilon$ is dominant, but $\eta +\varepsilon-\delta$ is not, then there exists a $j$ such that $\eta_j + \varepsilon_j = \eta_{j+1} +\varepsilon_{j+1}$ and $\delta_j=1,\delta_{j+1}=0$, which are the $j^{th}$ and $(j+1)^{th}$ row in  $\det|[x_i]^{\eta_j +n-j +1 + \varepsilon_j-\delta_j}| $ are identical, hence equal to 0.

Let $\mathcal{C}$ be the collection of $(\varepsilon,\delta)$ such that $\eta + \varepsilon$ and $\eta +\varepsilon-\delta$ are both not dominant. As $\eta + \varepsilon$ is not dominant, there is a smallest $j$ such that $\eta_j = \eta_{j+1}$ and $\varepsilon_j=0,\varepsilon_{j+1}=1$. We define,
\begin{align*}
\mathcal{C}_j&=\{(\varepsilon,\delta)\in \mathcal{C} : j \text{ is the smallest such that }  \eta_j = \eta_{j+1}, \varepsilon_j=0,\varepsilon_{j+1}=1\},\\
\mathcal{C}_{j,1} &= \{(\varepsilon,\delta)\in \mathcal{C}_j : \delta_j = \delta_{j+1} \}, \\ 
\mathcal{C}_{j,2} &= \{(\varepsilon,\delta)\in \mathcal{C}_j : \delta_j \neq \delta_{j+1} \}. 
\end{align*}
The sets  $\mathcal{C}_{j,1}$ and $\mathcal{C}_{j,2}$ partition the whole $\mathcal{C}$. If $(\varepsilon,\delta) \in \mathcal{C}_{j,1}$, then  the $j^{th}$ and $(j+1)^{th}$ row in  $\det|[x_i]^{\eta_j +n-j +1 + \varepsilon_j-\delta_j}| $ are identical. Hence the determinant vanishes. 

If $(\varepsilon,\delta) \in \mathcal{C}_{j,2}$, let 
 $$\delta' = (\delta_2,\delta_3, \ldots, \delta_{j-1},\delta_{j+1},\delta_j, \ldots,\delta_n).$$
If $\xi = \eta +\varepsilon-\delta$ and $\xi' = \eta +\varepsilon-\delta'$, then 
 $$\xi_j = \eta_j-1,\xi_{j+1} = \eta_{j+1}+1 \,\text{ and } \,\xi_j' = \eta_j,\xi_{j+1}' = \eta_{j+1}.$$ 
The respective numerators in (\ref{eq:symplecticweylnumpieri}) have same value but opposite signs. This gives us

 $$\sum_{(\varepsilon,\delta) \in \mathcal{C}_{j,2}} (-t)^{|\varepsilon|+|\delta|-n+1}  \det|[x_i]^{\eta_j +n-j +1 + \varepsilon_j-\delta_j}| =0.$$
 Hence we get that
$$\Delta \times \left( S^{(k)} \bigchi_{\eta} \right) = S^{(k)} \sum_{(\varepsilon,\delta) \in \mathcal{A}} (-t)^{|\varepsilon|+|\delta|-n+1} \bigchi_{\eta +\varepsilon-\delta} . $$

For a given $(\varepsilon,\delta) \in \mathcal{A}$, we  count all $(\varepsilon',\delta') \in \mathcal{A}$ such that $\varepsilon-\delta = \varepsilon'-\delta'$. If $\varepsilon_i-\delta_i =1$, then  $\varepsilon_i=1, ~\delta_i=0$. Similarly, if $\varepsilon_i-\delta_i =-1$ then  $\varepsilon_i=0,~ \delta_i=1$. If $\varepsilon_i-\delta_i =0$, then $\varepsilon_i = \delta_i$  and the number of such $i's$ to be $m= n-1 - |\varepsilon-\delta| $. 

There are $2^m$ different choices of $(\varepsilon',\delta') \in \mathcal{A}$ 
based on the choice of $0$ or $1$ in the $m$ places. Out of the $m$ places, if $\varepsilon_i = \delta_i =1$ at $k$ many places, then there are  $\binom{m}{k}$ independent possibilities. All of them do occur (for this we get that $|\varepsilon|+|\delta|-n+1 = -m + 2k$). Hence we  get that the coefficient of $\bigchi_{\eta +\varepsilon-\delta}$ is given by
$$(-1)^{m} \left( t^{-m} + {\binom{m}{1}} t^{-m+2} + \cdots + {\binom{m}{m-1}} t^{m-2} + t^{m}  \right) = (-t-t^{-1})^m.$$
Putting $S^{(1)}= (t+t^{-1})$, and taking $\nu = \eta+ \varepsilon -\delta$, we see that $|\varepsilon -\delta| = |\nu-\eta|$. Hence we get that,
$$\Delta \times \left( S^{(k)} \bigchi_{\eta} \right)= S^{(k)}\sum_{\substack{\nu \text{ dominant} \\ \nu-\eta \in \{ \pm 1, 0\}^{n-1}}}  (-S^{(1)})^{n-1- |\nu-\eta|} \bigchi_{\nu},$$
  which  completes the proof.
\end{proof}
 
\subsection{Branching computation}

\begin{theorem}[Symplectic Branching rule]
Let $\lambda$ be a dominant weight for $G$. Let
$$\bigchi_{\lambda}|_{H} = \sum_{\mu} \sum_k m(\lambda, \mu, k) S^{(k)} \bigchi_{\mu}. $$
Then $m(\lambda, \mu) :=\sum_k m(\lambda, \mu, k) S^{(k)} $ is non zero if and only if 
$$\lambda_j \geq \mu_i \geq \lambda_{j+2} ,\, \, \text{ for } 1 \leq j \leq n-1,$$
(here $\lambda_{n+1}=0$). When the inequalities are satisfied, let 
$$x_1 \geq y_1 \geq x_2 \geq y_2 \geq \cdots \geq x_n \geq y_n,$$
be the non decreasing rearrangement of $\{\lambda_1, \ldots, \lambda_n, \mu_1, \ldots, \mu_{n-1},0 \}$. Then
$$m(\lambda, \mu) = \prod_{i=1}^n S^{(x_i -y_i)}.$$
\end{theorem}

We will prove the theorem by using the relative Weyl character formula and strong induction the weights ordered by the lexicographic ordering. By the relative Weyl character formula, we have 

\begin{equation}
\bigchi_{\lambda}|_{H} = \sum_{\mu} m(\lambda,\mu) \bigchi_{\mu}=\frac{\displaystyle\sum_{i=0}^{n-1} (-1)^{i} S^{(\lambda_{n-i}+i)}\bigchi_{\lambda^{(n-i)}}}{\displaystyle\sum_{i=0}^{n-1} (-1)^{i} S^{(i)}\bigchi_{\omega_{n-i-1}}}.
\end{equation}
By cross multiplying by the relative Weyl denominator, we get

\begin{equation}
\Delta \times \left( \sum_{\mu} m(\lambda,\mu) \bigchi_{\mu} \right)=\sum_{i=0}^{n-1} (-1)^{i} S^{(\lambda_{n-i}+i)}\bigchi_{\lambda^{(n-i)}}.
\end{equation}

Note that $m(\lambda,\mu)$ is a representation of $SL_2$. By the relative Pieri formula,  

\begin{equation} \label{eq13}
\sum_{i=0}^{n-1} (-1)^{i} S^{(\lambda_{n-i}+i)}\bigchi_{\lambda^{(n-i)}} \!\!= \! \sum_{\mu} m(\lambda,\mu)\! \left( \sum_{\substack{\nu \text{ dominant} \\ \nu-\mu \in \{ \pm 1, 0\}^{n-1}}} \!\!\!\!\!\!\!\!\! (-S^{(1)})^{n-1- |\nu-\mu|} \bigchi_{\nu}\right) .
\end{equation}

Let $\mu$ be the largest in the lexicographic ordering such that $m(\lambda,\mu) \neq 0$, then $\nu=\mu + (1,1,\ldots,1)$ is also dominant, and is the largest by lexicographic ordering present in the expansion in (\ref{eq13}). Equating the largest in the lexicographic ordering on the left and right side of equation (\ref{eq13}), we get that 
$$\mu = \lambda^{(n)} - (1,1,\ldots,1) = (\lambda_1, \lambda_2, \ldots, \lambda_{n-1}).$$ Hence,  $$m(\lambda,\mu) = S^{(\lambda_n)}.$$

\begin{defn*}
 For a dominant weight $\lambda$ of $Sp(2n)$ and $\xi$ a dominant weight of $Sp(2n-2)$, we say that  $\xi$ {\em doubly interlaces} $\lambda$ if
$$\lambda_j \geq \xi_j \geq \lambda_{j+2} \quad \text{for} \quad 1\leq j \leq n-1.$$    
\end{defn*}

We will show that if $\mu$ does not satisfy the given double interlacing condition, then the corresponding multiplicity is 0.

We set up the induction hypothesis by assuming that if $\xi \geq \mu$ in the lexicographic ordering and $\xi$ doubly interlaces $\lambda$ then 
\begin{equation}\label{eq:symplecticmultiplicity}
    m(\lambda, \xi) = \prod_i S^{(x_i-y_i)},
    \end{equation}
where $\{x_1,y_1, x_2,y_2, \ldots, x_n,y_n\}$ be the non-decreasing rearrangement of $\{\lambda_1, \ldots, \lambda_n,$ 
 $ \xi_1, \ldots, \xi_{n-1},0 \}$  as  mentioned in the theorem and $m(\lambda, \xi)=0$ otherwise. We have already shown the first step of the induction hypothesis.

We first show a weak double interlacing property for the weights $(\lambda, \mu)$ such that the representations $m(\lambda, \mu)$ are non-vanishing:
\begin{lemma}\label{lemma:weakdoublyinterlacing}
   Suppose $m(\lambda,\mu)\neq 0$. Then 
   $$\lambda_j \geq \mu_j \geq \lambda_{j+2}-2, \quad \text{for all} \,\,j,\,\,  1 \leq j \leq n-1.$$
\end{lemma}

\begin{proof}
    Suppose there is an $k$ such that $\mu_k > \lambda_k$. Let $\nu=\mu+ (1,1,\ldots,1) $, then $\nu \neq \lambda^{(j)}$ for any $j$. We look at the coefficients of  $\bigchi_{\nu}$ in (\ref{eq13}). By linear independence of characters, we get that
\begin{equation}\label{eq-symplecticnoninterlacing}
0 = \sum_{\substack{\eta \text{ dominant} \\ \nu-\eta \in \{ \pm 1, 0\}^{n-1}}} (-S^{(1)})^{n-1- |\nu-\mu|}  m(\lambda, \eta).
\end{equation}
Note that $\eta_i - \nu_i = \eta_i - \mu_i -1 \in \{\pm 1,0\}$. Hence  for all $i$,  
$$\eta_i \in \{\mu_i,\mu_i+1,\mu_i+2\} .$$
For $\eta \neq \mu$ occurring in equation (\ref{eq-symplecticnoninterlacing}), we get that   $\mu < \eta$ in lexicographic ordering and  $\eta_k \geq \mu_k > \lambda_k$. This tells us that $\eta$ does not doubly interlace $\lambda$. By the induction hypothesis, for $\eta \neq \mu$, $m(\lambda, \eta)$ vanishes. Hence   equation (\ref{eq-symplecticnoninterlacing}) reduces to $(-S^{(1)})^{n-1}m(\lambda,\mu)=0$, which gives us that $m(\lambda,\mu)=0$.

The proof is similar in the case where $\mu_k <\lambda_{k+2}-2$ for some $k$.
\end{proof}

The proof of branching depends on whether the weight $\mu $ is of the form $\lambda^{(k)} -(1,1,\ldots,1)$ for some $k$, where $\lambda^{(k)}$ are the weights occurring in the numerator of the relative Weyl character formula. From the weak interlacing property given by the foregoing lemma, we break the proof into three cases: 
\begin{itemize}
    \item[1.] (\emph{Generic Case}) For some $j$ 
    $$\lambda_j-2 \geq \mu_j \geq \lambda_{j+2}-2,$$
    and $\mu \neq \lambda^{(k)} -(1,1,\ldots,1)$ for any $k$.
    
    \item[2.] (\emph{Boundary case}) For all $i$ 
    $$\lambda_i \geq \mu_i \geq \lambda_i -1.$$ 
    The weight $(\lambda_1,\lambda_2,\ldots, \lambda_{n-1}) = \lambda^{(n)} -(1,1,\ldots,1) $ is the base case of the induction hypothesis, which we have already handled before. Hence, we can assume that in the boundary case for some $i$, $\mu_i<\lambda_i$.
    
    \item[3.] (\emph{Weyl numerator case}) $\mu = \lambda^{(k)} -(1,1,\ldots,1)$ for some $k<n$.
\end{itemize}

We have the following key technical lemma which handles the generic and boundary cases:  \begin{lemma}\label{lemma_symplmultsum}
Suppose  $\mu$ is a weight such that for some $j, ~1\leq j <n$, 
\begin{align*}
     \lambda_j-1 &\geq \mu_j \geq \lambda_{j+2}-2,\\
     \text{ and} \quad \mu &\neq \lambda^{(k)} -(1,1,\ldots,1) \,\, \text{for any $k$.}  
\end{align*}
    Let $\varepsilon = (\varepsilon_1,\ldots, \varepsilon_{n-1}) \in \{0,1,2\}^{n-1}$ be such that $\varepsilon$ is non-zero and $\varepsilon_j=0$. Then, 
    \begin{equation}\label{eqn:keylemma-symp}
        m(\lambda,\mu+ \varepsilon) + m(\lambda,\mu+ \varepsilon+2 e_j) -S^{(1)} m(\lambda,\mu+ \varepsilon + e_j) =0,
    \end{equation}
    where $e_j$ is the standard $j^{th}$ coordinate $(0,\ldots,0,1,0,\ldots,0).$
\end{lemma}
\begin{proof}
 As in the statement of the branching law, let $\{x_1,y_1, x_2,y_2, \ldots, x_n,y_n\}$ be the non-decreasing rearrangement of $\{\lambda_1, \ldots, \lambda_n, \mu_1+\varepsilon_1, \ldots, \mu_{n-1}+\varepsilon_{n-1},0 \}$. Note that $\mu +\varepsilon$ is  larger that $\mu$ in lexicographic ordering. By the induction hypothesis given by equation (\ref{eq:symplecticmultiplicity}), the lemma follows if $\mu +\varepsilon$ is not dominant or not doubly interlacing $\lambda$, since 
    $$m(\lambda,\mu+ \varepsilon) = m(\lambda,\mu+ \varepsilon+2 e_j) = m(\lambda,\mu+ \varepsilon + e_j)=0. $$

Suppose $\mu+ \varepsilon$ is dominant and doubly interlaces $\lambda$ (so $\mu_j \geq \lambda_{j+2})$. We will break the proof into 4 cases.\\

\begin{enumerate}

\item[Case 1:] \,\,$\mu_j \leq \lambda_j-2$, $\mu_j \neq \lambda_{j+1}-1$ and $\mu_j =y_k$ for some $k$.

By the induction hypothesis given by equation (\ref{eq:symplecticmultiplicity}), 
\begin{align*}
    m(\lambda,\mu+\varepsilon) &= \left( \prod_{i\neq k}S^{(x_i-y_i)} \right) S^{(x_k-\mu_j)}\\
    m(\lambda,\mu+\varepsilon+e_j)&= \left( \prod_{i\neq k}S^{(x_i-y_i)} \right) S^{(x_k-\mu_j-1)}\\
    m(\lambda,\mu+\varepsilon+2e_j)&= \left( \prod_{i\neq k}S^{(x_i-y_i)} \right) S^{(x_k-\mu_j-2)}
\end{align*}
Upon substituting it in the left side of equation (\ref{eqn:keylemma-symp}), we get
\begin{align*}
    & m(\lambda,\mu+ \varepsilon) + m(\lambda,\mu+ \varepsilon+2 e_j) -S^{(1)} m(\lambda,\mu+ \varepsilon + e_j),\\
    =&\left( \prod_{i\neq k}S^{(x_i-y_i)} \right)(S^{(x_k-\mu_j)} + S^{(x_k-\mu_j -2)} -S^{(1)}S^{(x_k-\mu_j-1)} ),\\
    =& \left( \prod_{i\neq k}S^{(x_i-y_i)} \right) ( S^{(1)}S^{(x_k-\mu_j-1)} -S^{(1)}S^{(x_k-\mu_j-1)}).\\
    =&\,0.
\end{align*}

\item[Case 2:]  \,\,$\mu_j \leq \lambda_j-2$, $\mu_j \neq \lambda_{j+1}-1$ and $\mu_j =x_k$ for some $k$. 

This is similar to Case 1. Here $\mu_j = x_k$ for some $k$, which can be handled in a similar manner. \\

\item[Case 3:] $\mu_j =\lambda_j-1$.

When $\mu_j =\lambda_j-1$, one sees that $\mu + 2e_j$ is not doubly interlacing but larger than $\mu$ in lexicographic order, giving us $m(\lambda,\mu + 2e_j)=0$. Since $x_j = \lambda_j$, we have $x_j-\mu_j=1$. Hence, 
\begin{align*}
    & m(\lambda,\mu+ \varepsilon) -S^{(1)} m(\lambda,\mu+ \varepsilon + e_j)\\
    =&\left( \prod_{i\neq j}S^{(x_i-y_i)} \right)(S^{(x_j-\mu_j)} -S^{(1)}S^{(x_j-\mu_j-1)} )\\
    =& \left( \prod_{i\neq j}S^{(x_i-y_i)} \right) ( S^{(1)}-S^{(1)}S^{(0)})\\
    =&\,0.
    \end{align*}
    
\item[Case 4:] $\mu_j =\lambda_{j+1}-1$.

In this case we have $\lambda_{j+1} -1=\mu_j \geq \mu_{j+1}$.  By the above cases, we can work with $\mu_{j+1}$ instead of $\mu_j$ provided $\mu_{j+1} \neq \lambda_{j+2}-1$.  This process can be continued, unless in the extreme case,  we have where either $\mu_i=\lambda_i$ or $\mu_i=\lambda_{i+1}-1$ for all $i$. 

Let $k$ be the smallest index such that  $\mu_i = \lambda_{i}$  for all $i<k$, and $\mu_k=\lambda_{k+1}-1$. Also $\mu_{k+1} \leq \mu_k = \lambda_{k+1}-1<\lambda_{k+1}$, hence $\mu_{k+1} =\lambda_{k+2}-1$. Proceeding in a similar manner, we get  that for all $i\geq k$, $\mu_i=\lambda_{i+1}-1$. Hence,  $\mu = \lambda^{(k)}- (1,\ldots,1)$, which contradicts the hypothesis of the lemma (this will be considered in the Weyl numerator case,  to be handled later). \\

This completes the proof of the lemma. 
\end{enumerate}
\end{proof}

\subsubsection{Branching: Generic case.}
We now proceed to the proof of Theorem \ref{theorem:branching-symplectic}. We first handle the generic case. For computing the multiplicity of a dominant weight $\mu$ in this case, let $\nu=\mu+ (1,1,\ldots,1) $. We look at the coefficients of $\bigchi_{\nu}$ in (\ref{eq13}). By linear independence of characters, we get that
\begin{equation}\label{eq4}
\sum_{\substack{\eta \text{ dominant} \\ \nu-\eta \in \{ \pm 1, 0\}^{n-1}}} (-S^{(1)})^{n-1- |\nu-\eta|}  m(\lambda, \eta)=0.
\end{equation}

We assume that we are in the hypothesis of Lemma \ref{lemma_symplmultsum}, i.e.,  $\lambda_j-2\geq \mu_j \geq \lambda_{j+2}-2$, for some $j$ and  $\mu_j \neq \lambda_{j+1}-1$. Let 
\begin{align*}
    R_j &=\{ \varepsilon \in \{0,1,2 \}^{n-1}\mid \varepsilon_j=0\},\\
    r(\varepsilon)&= n-|\varepsilon - (1,1,\ldots,1)| -1.
\end{align*}

By equation (\ref{eq4}), the following sum vanishes: 

\begin{equation}\label{eq_symplecticrearranged}
 \sum_{\varepsilon \in R_j} (-S^{(1)})^{r(\varepsilon)} \left(m(\lambda,\mu+ \varepsilon) + m(\lambda,\mu+ \varepsilon+2 e_j) -S^{(1)} m(\lambda,\mu+ \varepsilon + \!e_j)\!\right).
\end{equation}

By Lemma \ref{lemma_symplmultsum},  for all non-zero $\varepsilon \in R_j$, the corresponding summands in equation (\ref{eq_symplecticrearranged}) vanish. The only summand remaining is the one corresponding to $\varepsilon = (0,\ldots,0)$, which gives us
\begin{equation}\label{eq_sympmultcomp}
    m(\lambda,\mu) = S^{(1)} m(\lambda,\mu+e_j) - m(\lambda,\mu+2e_j).
\end{equation}

We recall by induction hypothesis, that for all $\eta$ larger than $\mu$ in the lexicographic ordering, the multiplicity $m(\lambda,\eta)$ vanishes unless $\eta $  doubly interlaces $\lambda$. Suppose there is an $i\neq j$  such that $\mu_i \in\{\lambda_{i+2}-1,\lambda_{i+2}-2 \}$. Then $\mu +e_j$ and $\mu+2e_j$ do not doubly interlace $\lambda$. By equation (\ref{eq_sympmultcomp}), we have 
$$m(\lambda,\mu) = S^{(1)} m(\lambda,\mu+e_j) - m(\lambda,\mu+2e_j)=0+0=0.$$

Suppose $\mu_j$ is the only component of $\mu$ such that $\mu_j \in\{\lambda_{j+2}-1,\lambda_{j+2}-2 \}$. Then $\mu_{j+1}\leq \mu_j \leq \lambda_{j+2}-1$, and by the forgoing part, we can assume that $\lambda_{j+2} -1 \geq \mu_{j+1} \geq \lambda_{j+3}$. We consider $\mu_{j+1}$, and apply the forgoing argument, to conclude that $m(\lambda,\mu)$ vanishes. Hence we can assume that if $m(\lambda,\mu)$ does not vanish, then $\mu$ doubly interlaces $\lambda$.

Let  $\{x_1,y_1, x_2,y_2, \ldots, x_n,y_n\}$ be the non-decreasing rearrangement of $\{\lambda_1, \ldots, \lambda_n,$ $ \mu_1, \ldots, \mu_{n-1},0 \}$. Then either $\mu_j =y_k$ for some $k$ or $\mu_j = x_l$ for some $l$. We consider $\mu_j =y_k$, the other case follows in a similar  manner:
\begin{align*}
    m(\lambda,\mu)&=S^{(1)} m(\lambda,\mu+e_j) - m(\lambda,\mu+2e_j)\\
    &=\left( \prod_{i\neq k}S^{(x_i-y_i)} \right)(S^{(1)}S^{(x_k-\mu_j-1)} - S^{(x_k-\mu_j -2)}  )\\
    &= \left( \prod_{i\neq k}S^{(x_i-y_i)} \right) ( S^{(x_k-\mu_j)} + S^{(x_k-\mu_j -2)} - S^{(x_k-\mu_j -2)} )\\
    &=\left( \prod_{i\neq k}S^{(x_i-y_i)} \right) ( S^{(x_k-\mu_j)})\\
    &= \prod_{i}S^{(x_-y_i)}.
\end{align*}
 This completes the proof in the generic case.

\subsubsection{Branching: Boundary case.}

We recall that in the boundary case $\lambda_i \geq \mu_i \geq \lambda_i-1$ for all $i$. As discussed in the boundary case, we can assume that there is a $j$ such that $\mu_j=\lambda_j-1$. Let 
\begin{align*}
    R'_j &=\{ \varepsilon \in \{0,1\}^{n-1}\mid \varepsilon_j=0\}\\
    r'(\varepsilon)&= n-|\varepsilon - (1,1,\ldots,1)| -1
\end{align*}
Hence equation (\ref{eq4}) can be seen as:

\begin{equation}\label{eq5}
\sum_{\varepsilon \in R'_j} (-S^{(1)})^{r'(\varepsilon)} (m(\lambda,\mu+ \varepsilon)  -S^{(1)} m(\lambda,\mu+ \varepsilon + e_j))=0.
\end{equation}

By lemma \ref{lemma_symplmultsum}, we see that all summands corresponding to  $\varepsilon \neq (0,\ldots,0)$ vanish, which gives us that  
\begin{equation}
 m(\lambda,\mu) - S^{(1)} m(\lambda,\mu+e_j) = 0   
\end{equation}

Let  $\{x_1,y_1, x_2,y_2, \ldots, x_n,y_n\}$ be the non-decreasing rearrangement of $\{\lambda_1, \ldots, \lambda_n,$ 
 $ \mu_1, \ldots, \mu_{n-1},0 \}$. Then $x_i=\lambda_i$ and $y_i=\mu_i$ if $i<n$ and $y_n=0$. Hence for $i<n$, $x_i-y_i =1$ if $\mu_i=\lambda_i-1$, and vanishes otherwise. Note that $x_n-y_n=\lambda_n$.

Let $g(\mu) =  \# \{i| \mu_i =\lambda_i-1\}$.  By induction hypothesis, $m(\lambda,\mu+e_j) = S^{(\lambda_n)}(S^{(1)})^{g(\mu)-1} $. Hence by above equation, we get that $m(\lambda,\mu) = S^{(\lambda_n)}(S^{(1)})^{g(\mu)}$ which completes the proof in the boundary case.\\
 
\subsubsection{Branching: Weyl numerator case.}

As mentioned, $j=n$ was done as the base case for induction, hence we take $j\leq n-1$. We break the proof into two lemmas:

\begin{lemma}\label{lemma_sympseconftensorcomp}
If $a>b$ are integers, then 
\begin{equation}\label{eq_sympseconftensorcomp}
 S^{(a-b)} S^{(b-1)} - S^{(1)}S^{(a-b)}S^{(b)} + S^{(a-b-1)}S^{(b)} = -S^{(a+1)} .  
\end{equation}
\end{lemma}

\begin{proof}
We simplify the above equation:
\begin{align*}
    &\qquad\!\!\! S^{(a-b)} S^{(b-1)}  +(S^{(a-b-1)}- S^{(1)}S^{(a-b)})S^{(b)} \\
    &= S^{(a-b)} S^{(b-1)} +(S^{(a-b-1)}- S^{(a-b-1)} -S^{(a-b+1)}) S^{(b)}\\
    &= S^{(a-b)} S^{(b-1)} -S^{(a-b+1)}S^{(b)}.
\end{align*}
 By expanding the last tensor product, we get the desired result.
\end{proof}

The following lemma is crucial to the calculation of the multiplicities $m(\lambda,\mu):$
\begin{lemma}\label{lemma_symplastcomp}
Suppose for some $j$, 
$$\mu = \lambda^{(j)} -(1,1,\ldots,1) =(\lambda_1,\lambda_2,\ldots, \lambda_{j-1}, \lambda_{j+1}-1,\ldots,\lambda_{n}-1).$$  
For  $\varepsilon \in \{0,1,2\}^{n-1}$, define 
$$m'(\lambda,\mu+\varepsilon) = \begin{cases}
     m(\lambda,\mu+\varepsilon)  & \quad \text{if } \varepsilon \neq (0,0,\ldots,0)\\
     \prod_i S^{(x_i-y_i)}  & \quad  \text{if } \varepsilon = (0,0,\ldots,0), \\
\end{cases}$$
where $\{x_1,y_1, \ldots, x_n,y_n\}$ is the non-decreasing rearrangement of $\{\lambda_1, \ldots, \lambda_n,$ $ \mu_1, \ldots, \mu_{n-1},0 \}$. Then, 
    \begin{equation}\label{eq_symplspecialcasesum}
     \sum_{\substack{\mu +\varepsilon \text{ dominant} \\ \varepsilon \in \{0,1,2\}^{n-1}}} (-S^{(1)})^{n-1-|\varepsilon - (1,\ldots,1)|}  m'( \mu+\varepsilon) =(-1)^{n-j}S^{(\lambda_{j}+n-j)},
    \end{equation}
\end{lemma}
\begin{proof}
The proof proceeds by downwards induction on $j$. The case $j=n$ is the base case of induction. Suppose $\mu + \varepsilon$ doubly interlaces $\lambda$.  Then $\varepsilon_i=0$ for all $i<j$. Hence we can rewrite equation (\ref{eq_symplspecialcasesum}) as:
\begin{equation}
\!\!\!\!\!\!\!\sum_{\varepsilon_j}  \!\!\cdots \!\!\sum_{\varepsilon_{n-2}} (-S^{(1)})^{r(\varepsilon')} \!\!\left( m'(\mu+\varepsilon') + m'(\mu+\varepsilon'\!\!+\!\!2e_{n-1}) -S^{(1)}m'(\mu+\varepsilon'\!\! +\! e_{n-1}) \!\right)\!\!,
\end{equation}
where 
\begin{align*}
\varepsilon'& = (0,\ldots,0, \varepsilon_j, \varepsilon_{j+1},\ldots, \varepsilon_{n-2},0),\\
r(\varepsilon')&= n-|\varepsilon' - (1,1,\ldots,1)| -1.
\end{align*}
By the induction hypothesis and our assumption, we can write 
$$m'(\mu+\varepsilon')=\prod_i S^{(x'_i-y'_i)},$$ where  $\{x_1',y_1', \ldots, x_n',y_n'\}$ is the non-decreasing rearrangement of $\{\lambda_1, \ldots, \lambda_n, \mu_1+ \varepsilon_1', \ldots, \mu_{n-1}+\varepsilon_{n-1}',0 \}$). 
Using this, we rewrite the above sum as
\begin{align*}
   \sum_{\varepsilon'} &(-S^{(1)})^{r(\varepsilon')} \left(\prod_i^{n-3} S^{(x'_i-y'_i)} \right) \times \\ 
   & \left(S^{(x'_{n-2}-\lambda_n)}S^{(\lambda_n-1)} + S^{(x'_{n-2}-\lambda_n-1)}S^{(\lambda_n)} 
    -S^{(1)}S^{(x'_{n-2}-\lambda_n)}S^{(\lambda_n)}\right). 
\end{align*}
By lemma \ref{lemma_sympseconftensorcomp}, 
\begin{equation*}
    S^{(x'_{n-2}-\lambda_n)}S^{(\lambda_n-1)} + S^{(x'_{n-2}-\lambda_n-1)}S^{(\lambda_n)}
     -S^{(1)}S^{(x'_{n-2}-\lambda_n)}S^{(\lambda_n)}=-S^{(x'_{n-2}+1)}.
\end{equation*}
If $j=n-1$, then $x'_{n-2}=\lambda_{n-1}$ and we are done. If $j<n-1$, then $x'_{n-2}=\lambda_{n-1}-1$ if $\varepsilon'_{n-2}=0$, and equal to $\lambda_{n-1}$ otherwise. With this substitution, we see that

\begin{align*}
   \sum_{\varepsilon''} &(-S^{(1)})^{r(\varepsilon'')} \left(\prod_i^{n-4} S^{(x'_i-y'_i)}\right)\times\\
   &\left( S^{(x'_{n-3}-\lambda_{n-1})}S^{(\lambda_{n-1})} + S^{(x'_{n-3}-\lambda_{n-1}-1)}S^{(\lambda_{n-1}+1)}
     -S^{(1)}S^{(x'_{n-3}-\lambda_{n-1})}S^{(\lambda_{n-1}+1)} \right)
\end{align*}
where $\varepsilon''= (0,\ldots,0, \varepsilon_j, \varepsilon_{j+1},\ldots, \varepsilon_{n-3},0,0)$. 
By lemma \ref{lemma_sympseconftensorcomp}, 
\begin{align*}
& S^{(x'_{n-3}-\lambda_{n-1})}S^{(\lambda_{n-1})} + S^{(x'_{n-3}-\lambda_{n-1}-1)}S^{(\lambda_{n-1}+1)} -S^{(1)}S^{(x'_{n-3}-\lambda_{n-1})}S^{(\lambda_{n-1}+1)}\\
&= - S^{(x'_{n-3}+2)}.
\end{align*}
If $j=n-2$, then $x'_{n-3}=\lambda_{n-2}$, and we are done. Otherwise, we proceed in a similar manner for $n-j$ times, and apply lemma \ref{lemma_sympseconftensorcomp} every time. We would see that $x_{n-j-1}=\lambda_{n-j}$, which would give us the desired equality.
\end{proof}

\begin{corollary}
    Let $\mu = \lambda^{(j)} -(1,1,\ldots,1) $  for some $j$, then 
    $$m(\lambda,\mu) = \prod_i S^{(x_i-y_i)},$$
    where $\{x_1,y_1, \ldots, x_n,y_n\}$ is the non-decreasing rearrangement of $\{\lambda_1, \ldots, \lambda_n,$ $ \mu_1, \ldots, \mu_{n-1},0 \}.$
\end{corollary}
\begin{proof}
    For $\mu = \lambda^{(j)}-(1,1,\ldots,1) =(\lambda_1,\ldots,\lambda_{j-1}, \lambda_{j+1} -1, \ldots ,\lambda_n-1) $,  we look at the coefficients of $\bigchi_{ \lambda^{(j)}}$ in (\ref{eq13}), which gives us
\begin{equation}\label{eq_sympcasesp}
 (-1)^{n-j}S^{(\lambda_{j}+n-j)} = \sum_{\substack{\mu +\varepsilon \text{ dominant} \\ \varepsilon \in \{0,1,2\}^{n-1}}} (-S^{(1)})^{r(\varepsilon)}  m(\lambda, \mu+\varepsilon),   
\end{equation}
We rearrange the sum, to get
$$m(\lambda,\mu) = (-1)^{n-j}S^{(\lambda_{j}+n-j)} -  \sum_{\substack{\mu +\varepsilon \text{ dominant} \\ \varepsilon \in \{0,1,2\}^{n-1} \\ \varepsilon \neq (0,\ldots,0)}} (-S^{(1)})^{r(\varepsilon)}  m(\lambda, \mu+\varepsilon).$$

By Lemma \ref{lemma_symplastcomp}, we see that the right side of the equation is given by $\prod_i S^{(x_i-y_i)}$, which completes the proof.
\end{proof}

With this we have given the proof for the multiplicity formula in the three cases, hence this completes the proof of branching rules in the symplectic case.

 \begin{rmk}
 \normalfont
One can arrive at the classical result of the multiplicity, which is seen as a number, by considering the dimension of the $SL_2$ representation $m(\lambda,\mu).$ It follows canonically that if $m(\lambda,\mu)=\prod_{i=1}^n S^{(x_i -y_i)}$, then $\dim m(\lambda,\mu) = \prod_{i=1}^n (x_i -y_i+1)$, which is the formula seen in the classical treatment of the subject.
\end{rmk}

\section{Branching rules for $(Spin(2n) , Spin(2n-1))$}
When $(G,H) = (Spin(2n),Spin(2n-1))$, we are not in an equal rank case. In this situation, we do not seem to have at hand a relative Weyl character formula and the foregoing method cannot be applied. 
The main idea of our proof is to relate the Weyl character formula for $G$ to that of $GL(n)$. The Weyl group of $Spin(2n)$ is a semi-direct product of the symmetric group $S_n$ with the group of `even sign changes' $(\mathbb{Z}/2\mathbb{Z})^{n-1}\subset (\mathbb{Z}/2\mathbb{Z})^{n}$. The starting point of our proof is to decompose the Weyl character of $Spin(2n)$ in terms of the Weyl character formula for $GL(n)$, indexed by the sign changes. We then use the branching formula for $(GL(n), GL(n-1))$ as an algebraic identity and stitch the resultant identities along the group $(\mathbb{Z}/2\mathbb{Z})^{n-1}$ to arrive at a proof of the branching formula in this case.

We recall the convention we had in Section \ref{section_branchingspinodd}:

Given a tuple $\eta= (\eta_1, \eta_2  \ldots, \eta_{n})$, we define

$$D^+(\eta)= \det |x_j^{\eta_i} + x_j^{-\eta_i}| \quad\text{ and } \quad D^-(\eta)= \det |x_j^{\eta_i} - x_j^{-\eta_i}| . $$

The dominant weights of $G$ are given by  $\lambda = (\lambda_1, \lambda_2,\ldots,\lambda_n)$, where
$\lambda_i $ are all integers or all half integers such that $\lambda_1 \geq \lambda_2 \geq \ldots \geq |\lambda_{n}|$. Without loss of generality, we may assume that $\lambda_n \geq 0$ (otherwise we can apply the outer automorphism of $G$, to make it non negative). We recall the Weyl character formula for $G$:

\begin{equation}
\bigchi_{\lambda} = \frac{\tfrac{1}{2}(D^-(\lambda + \rho_G) + D^+(\lambda + \rho_G))}{\tfrac{1}{2}D^+(\rho_G)},
\end{equation}
where $\rho_G= (n-1,n-2, \ldots,1,0)$ is the half sum of positive roots for $G$. (We put the $\tfrac{1}{2}$ in the numerator and denominator to keep the product formula for the Weyl denominator coherent.)

The Weyl denominator has a product expansion:

$$\frac{1}{2}D^+(\rho_G) = \frac{\displaystyle\prod_{1 \leq i<j \leq n} (x_i-x_j)(x_ix_j-1)}{(x_1 \ldots x_n)^{n-1}}.$$

We first express the numerator in the Weyl character formula of $Spin(2n)$ as a sum of $GL(n)-$type numerators:

\begin{proposition} \label{prop_detexp}
Given a tuple $\eta= (\eta_1, \eta_2  \ldots, \eta_{n})$, we have that

$$\frac{1}{2}(D^-(\eta) + D^+(\eta)) = \sum_{\substack{\varepsilon \in \{ \pm 1\}^n\\ (-1)^{\varepsilon} =1}} \det |x_j^{\varepsilon_i \eta_i}|, $$
where $(-1)^{\varepsilon}= \prod_i \varepsilon_i.$
\end{proposition}
\begin{proof}
 We use the multilinearity of the determinant and expand along the columns to get, 
\begin{align*}
&\det |x_j^{ \eta_i} +x_j^{ -\eta_i}| = \sum_{\varepsilon \in \{ \pm 1\}} \det |x_j^{\varepsilon_i \eta_i}| \quad\text{ and }\\
&\quad \det |x_j^{ \eta_i} -x_j^{ -\eta_i}| = \sum_{\varepsilon \in \{ \pm 1\}}(-1)^{\varepsilon} \det |x_j^{\varepsilon_i \eta_i}|.
\end{align*} 
 Adding the two equations  and dividing by 2 completes the proof of the lemma. 
\end{proof}

We recall (see \cite[p.378]{goodman1998}), that we can consider the maximal torus of $Spin(2n-1)$ as embedded inside the maximal torus of $Spin(2n)$ by letting $x_n=1$.  Using this specialization, we rewrite  the Weyl denominator of $Spin(2n)$ as follows: 
\begin{lemma}\label{lem_dnbnweyldeno}
 If $WD(G)$ (resp. $WD(H)$) is the weyl denominator of $G$ ( resp. $H$), then
 $$WD(G)|_{x_n=1} = WD(H) \prod_{i=1}^{n-1}(x_i^{1/2} - x_i^{-1/2}).$$
\end{lemma}

\begin{proof}
We manipulate the product formula for the Weyl denominator of $G$:
\begin{align*}
WD(G)|_{x_n=1} &= \displaystyle  \frac{\displaystyle\prod_{1\leq i <j \leq n-1} (x_i-x_j)(x_i x_j-1) \prod_{i=1}^{n-1}(x_i-1)^2}{(x_1x_2\ldots x_{n-1})^{n-1}}\\
&= \frac{\displaystyle\prod_{1\leq i <j \leq n-1} (x_i-x_j)(x_i x_j-1) }{(x_1x_2\ldots x_{n-1})^{n-2}}\prod_{i=1}^{n-1} \frac{(x_i-1)^2}{x_i}\\
&= \left(\frac{\displaystyle\prod_{1\leq i <j \leq n-1} (x_i-x_j)(x_i x_j-1) \prod_{i=1}^{n-1}(x_i^{1/2} - x_i^{-1/2})}{(x_1x_2\ldots x_{n-1})^{n-2}}\right) \left( \prod_{i=1}^{n-1}(x_i^{1/2} - x_i^{-1/2})\right)\\
 &=WD(H)\displaystyle \prod_{i=1}^{n-1}(x_i^{1/2} - x_i^{-1/2}).
\end{align*}
\end{proof}

We reformulate Theorem \ref{theorem:branching-gln} as a formal algebraic identity involving the variables $x_i$:

\begin{proposition}
Let $\lambda_i $ be all integers or all half integers such that $\lambda_1 \geq \lambda_2 \geq \ldots \geq \lambda_{n} \geq 0$. Then, 

\begin{equation}\label{branchingdetidentity}
 \left.\frac{\det|x_j^{\lambda_i + n-i}|}{\det|x_j^{ n-i}|} \right|_{x_n=1} = \sum_{\mu \preceq \lambda} \frac{\det|x_j^{\mu_i + n-i-1}|}{\det|x_j^{n-i-1}|},
\end{equation}
where $\mu \preceq \lambda$ refers to saying that $\mu$ interlaces $\lambda$ ($\lambda_1 \geq \mu_1 \geq \lambda_2\ \geq \cdots \geq \lambda_{n-1} \geq \mu_{n-1} \geq \lambda_n$). Note that the determinants on the left side have size $n \times n$ whereas those on the right side are of size $n-1 \times n-1$.  
\end{proposition}

\begin{proof}
If $\lambda_i$ are all integers, then the above identity follows from Theorem \ref{theorem:branching-gln}.

If $\lambda$ has parts that are all half integers, we consider equation (\ref{branchingdetidentity}) for $\lambda' = \lambda - (\tfrac{1}{2},\tfrac{1}{2},\ldots,\tfrac{1}{2}) $. By multiplying both sides of the said identity by $(x_1x_2 \ldots x_{n-1})^{1/2}$, we can show that the above identity holds when $\lambda_i$ is half integral. In this case the right side of equation (\ref{branchingdetidentity}) is summed over all the $\mu$ that are half integers and interlace $\lambda$.
\end{proof}

\begin{corollary}
    \begin{equation}\label{branchingdetidentitymod}
 \frac{(\det|x_j^{\lambda_i + n-i}|)|_{x_n=1}}{\prod_{i=1}^{n-1} (x_i -1)} = \sum_{\mu \preceq \lambda} \det|x_j^{\mu_i + n-i-1}|.
\end{equation}
\end{corollary}
\begin{proof}
    
The denominators in equation (\ref{branchingdetidentity}) are Vandermonde determinant.  Upon cross multiplying we get the above equality.

\end{proof}

\subsection*{Proof of Theorem \ref{thmbranchingdn}}

By Proposition \ref{prop_detexp} and Lemma \ref{lem_dnbnweyldeno}, we have 

\begin{equation}\label{restrictedweylsumdn}
\bigchi_{\lambda}|_{x_n=1} = \frac{1}{WD(H)}\sum_{\substack{\varepsilon \in \{ \pm 1\}^n\\ (-1)^{\varepsilon}=1}} \frac{\det |x_j^{\varepsilon_i (\lambda_i + n-i)}|_{x_n=1}}{\prod_{i=1}^{n-1}(x_i^{1/2} - x_i^{-1/2})}.
\end{equation}

We first observe that the following two abelian groups are isomorphic:
$$\{ \varepsilon \in \{ \pm 1\}^n \mid  (-1)^{\varepsilon}=1 \} \cong \{ \varepsilon' \in \{ \pm 1\}^{n-1} \}. $$
Given $\varepsilon = (\varepsilon_1,\ldots,\varepsilon_n) \in \{ \pm 1\}^n, (-1)^{\varepsilon}=1$ ,we have the injection $\varepsilon' = (\varepsilon_1,\ldots,\varepsilon_{n-1})$. 

To show that it is a surjection, for $\varepsilon' \in \{ \pm 1\}^{n-1}$, let $\varepsilon'_n = \prod_{i=1}^{n-1} \varepsilon'_i$. Then $\varepsilon = (\varepsilon'_1,\ldots,\varepsilon'_{n-1},\varepsilon'_n)\in \{ \pm 1\}^n$ and $(-1)^{\varepsilon} =1$. Our convention is that by $\varepsilon$ we mean an element on the left side of the above isomorphism, and $\varepsilon'$ refers to its isomorphic image on the right side.

For a fixed $\varepsilon \in \{ \pm 1\}^n, (-1)^{\varepsilon}=1$, let $y_i = x_i^{\varepsilon_i}$. Then, 
\begin{align}
 \prod_{i=1}^{n-1}(x_i^{1/2} - x_i^{-1/2}) &= \prod_{i=1}^{n-1}(y_i^{\varepsilon_i/2} - y_i^{-\varepsilon_i/2})\nonumber \\
 &= (-1)^{\varepsilon'} \prod_{i=1}^{n-1}(y_i^{1/2} - y_i^{-1/2}) \nonumber\\
 &= (-1)^{\varepsilon'} (y_1y_2\ldots y_{n-1})^{-1/2} \prod_{i=1}^{n-1}(y_i-1). \label{weyldenodn}
\end{align}
Using equation (\ref{branchingdetidentitymod}), we get that 
\begin{align}
 \frac{(\det |x_j^{\varepsilon_i (\lambda_i + n-i)}|)|_{x_n=1}}{\prod_{i=1}^{n-1}(x_i^{1/2} - x_i^{-1/2})} &= (-1)^{\varepsilon'} (y_1y_2\ldots y_{n-1})^{1/2}\frac{\det |y_j^{\lambda_i + n-i}|_{y_n=1}}{\prod_{i=1}^{n-1}(y_i-1)}\nonumber \\
 &= (-1)^{\varepsilon'} (y_1y_2\ldots y_{n-1})^{1/2}\sum_{\mu \preceq \lambda} \det|y_j^{\mu_i + n-i-1}| \nonumber \\
 &=(-1)^{\varepsilon'} \sum_{\mu \preceq \lambda} \det|y_j^{\mu_i + n-i-1/2}|\nonumber \\
 &= (-1)^{\varepsilon'} \sum_{\mu \preceq \lambda} \det|x_j^{\varepsilon_j(\mu_i + n-i-1/2)}|\nonumber \\
 &= (-1)^{\varepsilon'} \sum_{\mu \preceq \lambda} \det|x_j^{\varepsilon_j'(\mu_i + n-i-1/2)}|.
\end{align}

Combining this with equation (\ref{restrictedweylsumdn}), we get
\begin{align}
\bigchi_{\lambda}|_{x_n=1} &= \frac{1}{WD(H)} \sum_{\substack{\varepsilon \in \{ \pm 1\}^n\\ (-1)^{\varepsilon}=1}} \frac{\det |x_j^{\varepsilon_i (\lambda_i + n-i)}|_{x_n=1}}{\prod_{i=1}^{n-1}(x_i^{1/2} - x_i^{-1/2})} \nonumber\\
 &=\frac{1}{WD(H)} \sum_{\varepsilon' \in \{ \pm 1\}^{n-1}} (-1)^{\varepsilon'} \sum_{\mu \preceq \lambda} \det|x_j^{\varepsilon_j'(\mu_i + n-i-1/2)}|\nonumber\\
 &=\sum_{\mu \preceq \lambda} \sum_{\varepsilon' \in \{ \pm 1\}^{n-1}} (-1)^{\varepsilon'} \frac{\det|x_j^{\varepsilon_j'(\mu_i + n-i-1/2)}|}{WD(H)}\nonumber\\
 &= \sum_{\mu \preceq \lambda} \frac{\det|x_j^{(\mu_i + n-i-1/2)} - x_j^{-(\mu_i + n-i-1/2)}|}{WD(H)}\nonumber\\
 &= \sum_{\mu \preceq \lambda} \bigchi_{\mu}. 
\end{align}

This completes the proof of Theorem \ref{thmbranchingdn}.

\subsection*{\centering Acknowledgement}
The authors thank  Dibyendu Biswas, Sourav Ghosh, Sameer Kulkarni, Niladri Patra, Dipendra Prasad and  Manodeep Raha for their suggestions and many valuable discussions on the said topic of this paper. 

The authors especially thank Swarnava Mukhopadhayay for his constant support and suggestions. The second author expresses his gratitude to Ashoka University for the warm hospitality and  productive working atmosphere during their visit. The first author is indebted to TIFR for providing a great working atmosphere during his tenure there.

\bibliographystyle{alpha}
\bibliography{ref.bib}{}

\end{document}